\documentclass[letterpaper, 11pt,  reqno]{amsart}
\usepackage[margin=1.2in,marginparwidth=1.5cm, marginparsep=0.5cm]{geometry}

\usepackage{amsmath,amssymb,amscd,amsthm,amsxtra, esint, xcolor}

\usepackage[implicit=true]{hyperref}
\allowdisplaybreaks[2]

\usepackage{color}

\definecolor{gr}{rgb}   {0.,   0.69,   0.23 }
\definecolor{bl}{rgb}   {0.,   0.5,   1. }
\definecolor{mg}{rgb}   {0.85,  0.,    0.85}
\definecolor{yl}{rgb}   {0.8,  0.7,   0.}
\definecolor{or}{rgb}  {0.7,0.2,0.2}

\newtheorem{theorem}{Theorem} [section]

\newtheorem{lemma}[theorem]{Lemma}
\newtheorem{proposition}[theorem]{Proposition}
\newtheorem{remark}[theorem]{Remark}

\newcommand{\noi}{\noindent}
\newcommand{\Z}{\mathbb{Z}}
\newcommand{\R}{\mathbb{R}}

\newcommand{\T}{\mathbb{T}}

\newcommand{\Gdl}{\mathcal{G}_{\dl} }
\newcommand{\Tdl}{\mathcal{T}_{\dl} }
\newcommand{\Qdl}{\mathcal{Q}_{\dl} }

\newcommand{\hi}{\textup{hi}}
\newcommand{\HI}{\textup{HI}}
\newcommand{\lo}{\textup{lo}}
\newcommand{\LO}{\textup{LO}}

\newcommand{\too}{\longrightarrow}

\newcommand{\BO}{\text{\rm BO} }

\let\P= \undefined
\newcommand{\P}{\mathbf{P}}

\newcommand{\Q}{\mathbf{Q}}

\newcommand{\F}{\mathcal{F}}

\newcommand{\dl}{\delta}

\newcommand{\eps}{\varepsilon}
\newcommand{\kk}{\kappa}

\newcommand{\G}{\Gamma}
\newcommand{\ld}{\lambda}

\newcommand{\ft}{\widehat}

\newcommand{\wt}{\widetilde}
\newcommand{\cj}{\overline}
\newcommand{\dx}{\partial_x}
\newcommand{\dt}{\partial_t}

\newcommand{\ta}{\theta}

\renewcommand{\l}{\ell}

\newcommand{\les}{\lesssim}
\newcommand{\ges}{\gtrsim}

\newcommand{\jb}[1]
{\langle #1 \rangle}

\newcommand{\ind}{\mathbf 1}

\renewcommand{\S}{\mathcal{S}}

\newcommand{\M}{\mathcal{M}}

\newcommand{\N}{\mathbb{N}}
\newcommand{\NN}{\mathcal{N}}

\newtheorem*{ackno}{Acknowledgements}

\renewcommand{\H}{\mathcal{H}}

\def\sgn{\textup{sgn}}

\newcommand{\Id}{\textup{Id}}

\numberwithin{equation}{section}
\numberwithin{theorem}{section}

\makeatletter
\@namedef{subjclassname@2020}{\textup{2020} Mathematics Subject Classification}
\makeatother

\begin{document}
\baselineskip = 14pt

\title[Deep-water limit of ILW in  $L^2$]
{Deep-water limit of 
the intermediate long wave
equation in $L^2$}

\author[A.~Chapouto, G.~Li, T.~Oh, and D.~Pilod]
{Andreia Chapouto, Guopeng Li, Tadahiro Oh, and Didier Pilod}

\address{
Andreia Chapouto,  School of Mathematics\\
The University of Edinburgh\\
and The Maxwell Institute for the Mathematical Sciences\\
James Clerk Maxwell Building\\
The King's Buildings\\
Peter Guthrie Tait Road\\
Edinburgh\\ 
EH9 3FD\\
 United Kingdom}

\email{a.chapouto@ed.ac.uk}

\address{
Guopeng Li, School of Mathematics\\
The University of Edinburgh\\
and The Maxwell Institute for the Mathematical Sciences\\
James Clerk Maxwell Building\\
The King's Buildings\\
Peter Guthrie Tait Road\\
Edinburgh\\ 
EH9 3FD\\
 United Kingdom, 
 and
Department of Mathematics and Statistics, 
Beijing Institute of Technology, 
Beijing, China}

\email{guopeng.li@ed.ac.uk}

\address{
Tadahiro Oh, School of Mathematics\\
The University of Edinburgh\\
and The Maxwell Institute for the Mathematical Sciences\\
James Clerk Maxwell Building\\
The King's Buildings\\
Peter Guthrie Tait Road\\
Edinburgh\\ 
EH9 3FD\\
 United Kingdom}

\email{hiro.oh@ed.ac.uk}

\address{
Didier Pilod\\
Department of Mathematics, University of Bergen, Postbox 7800, 5020 Bergen, Norway}

\email{Didier.Pilod@uib.no}



\begin{abstract}
We study the well-posedness issue of
the intermediate long wave
equation (ILW)   on both the real line and  the circle.
By applying the gauge transform for the 
Benjamin-Ono equation (BO)
and
adapting the $L^2$ well-posedness argument for BO
by Molinet and the fourth author (2012), 
we prove global well-posedness of ILW
in $L^2$
on both the real line and the circle.
In the periodic setting, this provides
the first low regularity well-posedness of ILW.
We then establish convergence
of the ILW dynamics to the BO dynamics
in the deep-water limit
at the $L^2$-level.


\vspace{-10mm}

\end{abstract}

%
\maketitle

\tableofcontents

\baselineskip = 14pt

\section{Introduction}

\label{SEC:1}


In this paper, we study 
the  intermediate long wave equation (ILW):
\begin{align}
\begin{cases}
\dt u -
 \Gdl \partial_x^2 u  =  \dx (u^2)    \\ 
u|_{t = 0} = u_0,
\end{cases}
\qquad ( t, x) \in \R \times \M, 
\label{ILW1}
\end{align}

\noi
where
$\M =\R$ or $    \T = \R/(2\pi \Z)$.
The equation \eqref{ILW1},  introduced in \cite{JE, KKD}, 
 models 
the internal wave propagation of the interface 
in a stratified fluid of finite depth $\dl > 0$, 
and the unknown $u :\R\times \M \to \R$
denotes the amplitude of the internal wave at the interface.
See also Remark~1.1 in \cite{LOZ}.
The  operator $\Gdl $ in \eqref{ILW1}
is defined by
\begin{align}
\Gdl   = 
\Tdl - \dl^{-1} \dx^{-1},  
\label{OP0}
\end{align}

\noi
where $\Tdl$
is given by 
\begin{align}
\ft{\Tdl f}(\xi)  =  -  i \coth(\dl \xi ) \ft f(\xi), 
\quad \xi \in\ft \M.
 \label{OP0a}
\end{align}

\noi
Here, $\ft \M$ denotes 
the Pontryagin dual of $\M$, i.e.
$\ft \M = \R$ if $\M = \R$, 
and 
$\ft \M = \Z$ if $\M = \T$, 
and 
  $\coth$ denotes the usual hyperbolic cotangent function:
\begin{align}
\coth(x ) = \frac{   e^x + e^{-x}  }{ e^x - e^{-x}  }
= \frac{e^{2x}+1}{e^{2x}-1}, \quad x \in \R\setminus\{0\}
\label{OP1a}
\end{align}

\noi
with the convention $\coth(x) - \frac{1}{x} = 0$ for $x=0$.
Note that 
$\Tdl$ 
 is the inverse of the 
so-called Tilbert transform (with an extra $-$ sign), 
appearing in the study of water waves
of finite depth \cite{BLS, HIT, AIT}.
In particular, $\Tdl$
is related to 
the Dirichlet-to-Neumann map in a two-dimensional finite depth
two-layer fluid, 
and in the deep-water limit ($\dl \to \infty$), 
it converges to the Hilbert transform $\H$
(with the Fourier multiplier $-i \sgn (\xi)$), 
corresponding to 
the Dirichlet-to-Neumann map 
in infinite depth, 
which indicates
possible convergence of ILW~\eqref{ILW1}
to 
the Benjamin-Ono equation~(BO):
\begin{align}
\dt u -
 \H \partial_x^2 u  =  \dx (u^2).
\label{BO1}
\end{align}

\noi
Indeed, 
 ILW \eqref{ILW1} is an important physical model,  
 providing 
 a natural connection between 
the deep-water regime (= the Benjamin-Ono regime) and the shallow-water regime
(= the KdV regime).
As such, it has been studied extensively from both the applied and theoretical points of view.
See, for example,  
a recent book  \cite[Chapter~3]{KS2022} by  Klein and Saut
for an overview of the subject
and the references therein.
See also a survey \cite{Saut2019}.
While the ILW equation  has attracted increasingly more attention
in recent years
(see, for example,  \cite{MV, KPV20, MPS, Gli, Gli2, LOZ, IS, CLOZ}), 
these two references indicate that the rigorous mathematical study of 
ILW is still widely  open, 
especially as compared to BO and KdV.
In particular, one of the fundamental, but challenging questions
is the convergence properties
of ILW in the deep-water limit ($\dl\to \infty$)
and in the shallow-water limit 
($\dl \to 0$).
See \cite{Gli, Gli2, LOZ, CLOZ}
for recent progress on convergence problems
from both the deterministic and statistical points of view.

Our main goal in this paper is to present
a unified argument for proving $L^2$ global well-posedness
of ILW \eqref{ILW1} on both the real line and the circle
and then to establish convergence of the ILW dynamics
to the BO dynamics in the deep-water limit
at the $L^2$-level.
Before stating our main result, 
let us recall a few conservation laws for ILW \eqref{ILW1}:
\begin{align}
\begin{split}
 \text{mean: } & \int_\M u  dx, 
\qquad 
\text{mass: } M(u) = \int_\M u^2  dx, \\
 \text{Hamiltonian: }
& E_\dl(u) =
\frac 12   \int_\M  u \Gdl \dx u   dx
+ \frac{1}{3} \int_\M u^{3} dx.
\end{split}
\label{cons1}
\end{align}

\noi
In particular, ILW \eqref{ILW1}
is a Hamiltonian PDE.
Furthermore, it is known to 
 be completely integrable.
We, however, do not make use of the completely integrable structure
in this paper.

We now state our first result.

\begin{theorem}\label{THM:1}
Let $\M = \R$ or $\T$  and $0 < \dl < \infty$.
Then, given any $s \ge 0$, ILW \eqref{ILW1} is 
globally well-posed in $H^s(\M)$.
Moreover, when $\M = \T$, 
by restricting to  the subspace $H^s_0(\T)$
consisting  of mean-zero functions, 
the solution map for \eqref{ILW1}
is locally 
Lipschitz continuous
from $H^s_0(\T)$ to $C(\R; H^s_0(\T))$, 
where the latter is endowed with the compact-open topology \textup{(}in time\textup{)}.

\end{theorem}

In \cite{MST}, it was shown that 
 both BO and ILW
are quasilinear in the sense that 
the solution map is not smooth; see also \cite{KTz2}.
This  in particular implies
that we can not use a contraction argument 
for constructing solutions, making the  well-posedness question 
rather challenging.
Nonetheless, the  well-posedness issue of 
the BO equation \eqref{BO1}, corresponding to 
the infinite depth case ($\dl = \infty$), 
has been studied extensively
on both the real line and the circle;
see \cite{Iorio, Ponce, KTz1, TAO04, IK, Moli1, BP, Moli2, MP, MV, IT1, 
GKT, KLV}.
On the other hand, there are only a handful of known well-posedness
results 
for ILW \eqref{ILW1};
see \cite{ABFS, MV, MPV, IS}.
In \cite{MV}, 
Molinet and Vento proved 
local well-posedness of \eqref{ILW1} in $H^\frac 12(\M)$, 
$\M = \R$ or $\T$, 
via the improved energy method, 
which immediately extended
to global well-posedness thanks to the conservation laws
in~\eqref{cons1}.
In the real line case, 
Molinet, Vento, and the fourth author
\cite{MPV}
refined the improved energy method in \cite{MV}
(with the short-time Strichartz estimates
as in \cite{KTz1, KK, LPS}
and a generalized  Coifman-Meyer theorem) 
and proved local well-posedness of \eqref{ILW1}
in $H^s(\R)$ for $s > \frac 14$.
In a recent work \cite{IS}, 
Ifrim and Saut studied \eqref{ILW1} on 
the real line and proved its global well-posedness in $L^2(\R)$
(along with long-time behavior of small solutions), 
based on an adaptation of 
the argument  developed for the BO equation \eqref{BO1}
by Ifrim and Tataru~\cite{IT1}
which combines the BO gauge transform and the (quasilinear) norm form method.
We point out that the argument in \cite{IS}
relies on dispersive estimates
and thus is not applicable to the periodic setting.

Our main strategy is to view \eqref{ILW1}
as a perturbation\footnote{after applying the Galilean transform \eqref{gauge0}
to remove the term $\dl^{-1}\dx^{-1}$ in \eqref{OP0}. See \eqref{ILW2}.}
 of \eqref{BO1} as in \cite{IS}
and apply the BO gauge transform, originally introduced by Tao \cite{TAO04}.
We then adapt 
the argument in  \cite{MP}, 
where
Molinet and the fourth author
introduced  a simplified and unified
approach
to the $L^2$ well-posedness
of the BO equation \eqref{BO1}
on both the circle and real line, 
to the current ILW problem.
See Sections \ref{SEC:gauge} and \ref{SEC:LWP}
for details.

Next, we state our second result 
on convergence in the deep-water limit.

\begin{theorem}\label{THM:2}
Let $\M = \R$ or $\T$
and $s\ge 0$.
Given  $u_0 \in H^s(\M)$, 
let $u$ be the global solution to the BO equation \eqref{BO1}
with $u|_{t = 0} = u_0$ constructed in \cite{MP}.
Let  $\{u_{0, \dl}\}_{1 \le \dl <  \infty}\subset H^s(\M)$
such that 
$u_{0, \dl}$ converges to $u_0$ in $H^s(\M)$ as $\dl \to \infty$, 
and 
let $u_\dl$ be the global solution to~\eqref{ILW1} with 
$u_\dl|_{t = 0} = u_{0, \dl}$
constructed in Theorem \ref{THM:1}.
Then, as $\dl \to \infty$, 
$u_\dl$ converges to $u$
in 
 $C(\R; H^s(\M))$, 
endowed with the compact-open topology \textup{(}in time\textup{)}.

\end{theorem}

Theorem \ref{THM:2} follows
from a minor adaptation of the proof of Theorem \ref{THM:1}.
See Section \ref{SEC:conv}.

\medskip

In \cite{Gli}, the second author proved the deep-water convergence
on both the real line and the circle
for $s > \frac 12$.
In view of the known results \cite{MV, Gli}, 
we restrict our attention to the range
$0 \le s \le \frac 12$
in the remaining part of the paper.
Furthermore, we present the proofs of Theorems 
\ref{THM:1} and \ref{THM:2} only in the periodic case, 
which is closely related to the study of invariant measures;
see Remark \ref{REM:inv}.
The required modification for the real line case
is minor and straightforward in view of the presentation
of the current paper and \cite{MP}, where
the latter treated the real line case.
In the periodic setting, 
in order to avoid the problem at the zeroth frequency, 
it is crucial to work with mean-zero functions.
In view of the conservation of the spatial mean, 
we  perform a change of unknowns via
the following Galilean transform \cite{CKSTT2} (with $\mu (u_0) = \frac 1{2\pi}\int_\T u_0(x) dx$):
\begin{align}
\wt u(t, x) = \G_{u_0}(u) (t, x) = u (t, x- 2\mu(u_0) t) - \mu (u_0)
\label{gaugex}
\end{align}

\noi
to reduce the problem to the mean-zero case in proving Theorem \ref{THM:1}.
Lastly, for simplicity of the presentation, 
we restrict our attention to 
positive times in the remaining part of the paper.

\begin{remark}\rm
(i)
The local Lipschitz continuity 
the solution map for \eqref{ILW1}
on $\T$ claimed in Theorem~\ref{THM:1}
indeed holds 
on  any hyperplane of $H^s(\T)$ of functions with a fixed mean.
In \cite[Theorem 1.2]{Moli2}, 
Molinet proved that the solution map to \eqref{BO1} on $H^s_0(\T)$
is in fact real-analytic
for $s\ge 0$ (and fails to be smooth for $s < 0$).
While we expect the same results to hold for ILW \eqref{ILW1}, 
we do not pursue this issue in this paper.

\medskip

\noi
(ii) Proceeding as in \cite{MP}, 
we can prove 
unconditional uniqueness
of  
ILW \eqref{ILW1}
in $H^s(\R)$ for $ \frac 14 < s \le \frac 12$ and $H^\frac 12 (\T)$.
(For $s > \frac 12$, it was shown in \cite{MV} 
on both the real line and the circle.)
Since the required modification is straightforward, 
we omit details.

\end{remark}

\begin{remark}\label{REM:shallow}\rm
Given $\dl > 0$, 
consider  the following scaling transformation \cite{ABFS}:
\noi
\begin{align}
v(t,x) = \tfrac{3}{ \dl} u\big(\tfrac{3}{ \dl} t,x\big),
\label{sILW1}
\end{align}

\noi
which transforms \eqref{ILW1} to  the following scaled ILW equation:
\begin{align}
\dt v   -  \frac{3}{ \dl}  \Gdl   \dx^2 v= \dx(v^2).
\label{sILW2}
\end{align}

\noi
Then, under suitable assumptions, 
the scaled ILW~\eqref{sILW2} is known to  converge to the  KdV equation
in the shallow-water limit ($\dl\to 0$)
\cite{ABFS, Gli, LOZ}:
\begin{align*}
\dt v + \dx^3 v = \dx(v^2) . 
\end{align*}

\noi
It would be of interest to study the shallow-water convergence
of ILW \eqref{ILW1} at the $L^2$-level.
As mentioned above, our argument for proving Theorem \ref{THM:1} relies
on the Galilean transform~\eqref{gauge0} to remove $\dl^{-1}\dx^{-1}$
in~\eqref{OP0} (which is crucial in the shallow-water limit)
which is not amenable to the shallow-water limit
since the spatial translation by $\dl^{-1} t$ diverges as $\dl \to 0$.

We point out that the scaling transformation \eqref{sILW1} is consistent with the physical
setting.
The ILW equation \eqref{ILW1} describes the motion
of the fluid interface in a stratified fluid of depth $\dl > 0$,
where $u$ denotes the amplitude of the internal wave at the interface.
As $\dl \to 0$, the entire fluid depth tends to $0$
and, in particular, 
the amplitude of the internal wave at the interface is $O(\dl)$,  which also  tends to $0$.
Hence, if we want to observe any meaningful limiting behavior, 
we need to magnify the fluid motion by a factor $\sim \frac 1\dl$, which
is exactly 
what the scaling transformation \eqref{sILW1} does.

\end{remark}

\begin{remark}\label{REM:inv}\rm
The ILW  equation \eqref{ILW1}
and the BO equation \eqref{BO1} are known to be completely integrable.
In particular, they possess infinite sequences of conservation laws $E_{\dl, \frac k2}(u)$
and $E_{\BO, \frac k2}(u)$
 at  regularities
$\frac k2$, $k \in \N \cup\{0\}$.
In \cite{LOZ, CLOZ}, 
the first three authors (with Zheng)
constructed the (formally invariant) measures $\rho_{\dl, \frac k2}$
(with a density of the form $\exp(-  E_{\dl, \frac k2}(u))$
associated with these conservation laws for the periodic ILW on $\T$, $k \in \N$,
and proved their convergence 
to the corresponding measures for the periodic BO,  constructed by Tzvetkov and Visciglia \cite{TV1}, 
in the deep-water limit.\footnote{In \cite{LOZ, CLOZ}, 
convergence in the shallow-water limit was also studied, but we 
restrict our discussion to  the deep-water limit here.}
Moreover, when $k \ge 3$, 
they established invariance of these measures $\rho_{\dl, \frac k2}$ under the ILW dynamics 
(for each fixed $0 < \dl< \infty$)
and convergence of the associated invariant ILW dynamics
to the corresponding invariant BO dynamics \cite{TV2, TV3} in the deep-water limit.

It is known that 
$\rho_{\dl, \frac k2}$ is supported on $H^{\frac k2 - \eps}(\T)\setminus H^{\frac k2}(\T)$
for any $\eps > 0$.
In particular, when $k = 2$, 
the measure $\rho_{\dl, 1}$ is supported slightly outside $H^\frac 12 (\T)$, 
where  Theorem \ref{THM:1} now provides globally well-defined dynamics.
Therefore, the $k = 2$ case is now within reach.
In a forthcoming work, we will treat the $k = 2$ case:
invariance of the measure $\rho_{\dl, 1}$ under the ILW dynamics
(for each fixed $0 < \dl< \infty$)
and deep-water convergence of the invariant $\rho_{\dl, 1}$-dynamics to the 
corresponding invariant BO dynamics
due to Deng, Tzvetkov, and Visciglia~\cite{DTV}.

\end{remark}


\begin{remark}\rm
In a recent preprint \cite{CFLOP}, 
with J.~Forlano, we
further exploited the idea
of viewing ILW \eqref{ILW1}
as a perturbation of BO \eqref{BO1}
and established
a global-in-time a priori
bound on the $H^s$-norm of a solution to \eqref{ILW1}
for $ - \frac 12 < s < 0$.
More precisely, 
by making use of the complete integrability of the BO equation, 
we showed that, given $ - \frac 12 < s < 0$,  there exists 
an (explicit) non-decreasing function $C_s: \R_+ \times \R_+ \to \R_+$
such that 
\[ \| u(t) \|_{H^s} \leq C_s\big(|t|, \|u_0\|_{H^s}\big)\]

\noi
for 
 any smooth solution $u$ to \eqref{ILW1}
on 
$\M = \R$ or $\T$
and any $t \in \R$.
In the same work, we
also proved ill-posedness of \eqref{ILW1}
in $H^s(\M)$ for $s < -\frac 12$, 
thus establishing sharpness (of the regularity range) of the 
aforementioned a priori bound
(modulo the endpoint $s = - \frac 12$).
Despite the lack of scaling invariance, 
these results allow us to identify
 $s = -\frac 12$
as  a ``critical'' regularity for ILW;
see Remark~\ref{REM:scaling} below
for a discussion on the scaling property of ILW.

We also mention a recent work \cite{LP}, 
where the authors
 studied long time behavior of solutions 
to ILW \eqref{ILW1} on the real line, 
 in the same spirit 
 as the work  \cite{FLMP}
 on long time decaying properties of solutions to 
 the BO equation \eqref{BO1}.

\end{remark}


\section{Notations, function spaces, and preliminary 
estimates}
\label{SEC:2}

By $A\les B$, we mean  $A\leq CB$ 
for some  constant $C> 0$.
We use   $A \sim  B$
to mean 
 $A\les B$ and $B\les A$.
We write $A\ll B$, if there is some small $c>0$, 
such that  $A\leq cB$.  

Let $s \in \R$ and $1 \leq p \leq \infty$.
We define the $L^2$-based Sobolev space $H^s(\T)$
by the norm:
\begin{align*}
\| f \|_{H^s} =   \| \jb{n}^s \ft f (n) \|_{\l^2_n},
\end{align*}

\noi
where 
$\jb{x} = (1 + |x|^2)^\frac 12$, 
and we set $H^s_0(T) = \big\{ f \in H^s(\T): \ft f(0) = 0\big\}$
as the subspace consisting of mean-zero functions.
We also define the $L^p$-based Sobolev space $W^{s, p}(\T)$
by the norm:
\begin{align*}
\| f \|_{W^{s, p}} =   \|\jb{D}^s f\|_{L^p} = \big\| \F^{-1} (\jb{n}^s \ft f(n)) \big\|_{L^p}, 
\end{align*}

\noi
where 
$D = - i \dx$ (with the Fourier multiplier $n$)
and 
$\F^{-1}$ denotes the inverse Fourier transform.
When $p = 2$, we have $H^s(\T) = W^{s, 2}(\T)$. 
In dealing with space-time function spaces, 
we use short-hand notations such as
$L^p_T H^s_x$
for $L^p([0, T]; H^s(\T))$.

Let $\eta:\R \to [0, 1]$ be a smooth  bump function supported on $[-\frac{8}{5}, \frac{8}{5}]$ 
and $\eta\equiv 1$ on $\big[-\frac 54, \frac 54\big]$.
For $\xi \in \R$, we set $\varphi_0(\xi) = \eta(|\xi|)$
and 
\begin{align*}
\varphi_{j}(\xi) = \eta\big(\tfrac{|\xi|}{2^j}\big)-\eta\big(\tfrac{|\xi|}{2^{j-1}}\big)
\end{align*}

\noi
for $j \in \N$.
Then, for $j \in \Z_{\geq 0} := \N \cup\{0\}$, 
we define  the Littlewood-Paley projector  $\Q_j$ 
as the Fourier multiplier operator with a symbol $\varphi_j$.
Note that we have 
\begin{align*}
\sum_{j = 0}^\infty \varphi_j (\xi) = 1
\end{align*}

\noi
 for each $\xi \in \R$. 
Thus, 
we have the following Littlewood-Paley decomposition:
\[ f = \sum_{j = 0}^\infty \Q_j f. \]

Define the projectors $\P_+$ and $\P_-$ by setting
\[ \P_{+} f = \F^{-1} (\ind_{\{n \ge 1\}} \ft f(n)) 
\qquad \text{and}
\qquad 
\P_{-} f = \F^{-1} (\ind_{\{n \le -1\}} \ft f(n)) .\]

\noi
We also define $\P_\hi$, $\P_\HI$, $\P_\lo$, and $\P_\LO$
by setting 
\begin{align}
 \P_\hi = \sum_{j = 1}^\infty \Q_j, 
\quad 
\P_\HI = \sum_{j = 3}^\infty \Q_j, 
\quad \P_\lo = \Id - \P_\hi, 
\quad \text{and} \quad 
\P_\LO = \Id - \P_\HI, 
\label{proj1}
\end{align}

\noi
and set $\P_{\pm, \hi} = \P_{\pm} \P_\hi$, etc. 

Next, we
recall the $X^{s, b}$-spaces  introduced by 
Bourgain \cite{BO1, BO2} and their variants
defined by the following norms:
\begin{align*}
\| u\|_{X^{s, b}} & = \| \jb{n}^s \jb{\tau - |n|n}^b \ft u(\tau, n) \|_{\l^2_nL^2_\tau }, \\
\| u\|_{Z^{s, b}} & = \| \jb{n}^s \jb{\tau - |n|n}^b \ft u(\tau, n) \|_{\l^2_nL^1_\tau}, \\
\| u\|_{\wt Z^{s, b}} 
& =  \bigg(\sum_{j = 0}^\infty 
\| \Q_j u\|_{ Z^{s, b}}\bigg)^\frac 12, \\
\| u\|_{Y^{s, b}}
& = \| u\|_{X^{s, b}}
+ \| u\|_{\wt Z^{s, b- \frac 12 }}.
\end{align*}

\noi
See
\cite{BO1, BO2, KPV93b, CKSTT} for basic properties.
Note that we have $\wt Z^{s, b} \subset Z^{s, b}$.
For  $T>0$,  we define the space $X^{s,b}(T)$ to be the restriction of 
the $X^{s, b}$-space onto the time interval 
$[0, T]$ via the norm:
\begin{equation}
\| u \|_{X^{s,b}(T)} := \inf \big\{ \| v \|_{X^{s,b}}: v|_{[0, T]} = u \big\}.
\label{loc1}
\end{equation}

\noi
We also define the restriction spaces $\wt Z^{s, b}(T)$
and $Y^{s, b}(T)$
in an analogous manner.
Note that $Y^{s, b}(T)$ is complete.
Recall the following embedding:
\begin{align}
 Y^{s, \frac 12}(T) \subset  Z^{s,0}(T) \subset C([0, T]; H^s(\T)).
 \label{loc2}
\end{align}

Let $S(t)$ denote the linear propagator
for the BO equation \eqref{BO1} defined by 
\begin{align*}
\ft {S(t) f}(n) = e^{i t|n| n} \ft f(n).
\end{align*}

\noi
Then, we have the following linear estimates;
see
\cite{BO1, KPV93b, TAO, MP}.

\begin{lemma}\label{LEM:lin}
Let $s \in \R$.  Then, we have 
\begin{align*}
\| \eta(t) S(t) f \|_{Y^{s, \frac 12}} & \les \| f\|_{H^s},\\
\bigg\| \eta(t) \int_0^t S(t - t') 
F (t') dt'\bigg\|_{Y^{s, \frac 12}}
& \les \| F \|_{Y^{s, - \frac 12}}.
\end{align*}

\end{lemma}

Next, we recall the periodic $L^4$-Strichartz estimate
due to Bourgain \cite{BO1}; see also \cite[(6)]{Moli1}.

\begin{lemma}\label{LEM:L4}
The following estimate holds\textup{:}
\begin{align*}
\| u \|_{L^4_{t, x}} \les \|u \|_{X^{0, \frac 38}}
\les \|u \|_{Y^{0, \frac 12}}.
\end{align*}

\end{lemma}

For $0 < \dl \le \infty$, define $\Qdl$ by 
\begin{align}
\Qdl =( \Tdl - \H)\dx
\label{OP3}
\end{align}

\noi
with the understanding that $\mathcal{Q}_\infty = 0$
(corresponding to the BO equation \eqref{BO1}), 
where $\Tdl$ is as in \eqref{OP0a}.
The operator $\Qdl$ naturally appears
in viewing ILW as a perturbation of the BO equation;
see \eqref{ILW2} and \eqref{ILW4} below.
The next lemma establishes a smoothing property of 
the operator
$\Qdl$;
see also Lemma 2.2 in \cite{IS}.

\begin{lemma}\label{LEM:smooth}
Let $ s\ge 0$.
Then, we have
\begin{align}
\| \Qdl f \|_{H^s} \les \dl^{-1}(1+ \dl^{-s})\|f\|_{L^2}.
\label{smooth0}
\end{align}

\noi
In particular, we have 
\begin{align*}
\| \Qdl f \|_{H^s} \les \|f\|_{L^2}, 
\end{align*}

\noi
uniformly in  $\dl \ge 1$.

\end{lemma}

\begin{proof}
With a slight abuse of notation, 
let $\ft \Qdl(n)$ denotes the multiplier 
for the operator $\Qdl$:
\begin{align}
\ft \Qdl(n) = n \big(\coth (\dl n) - \sgn (n)\big).
\label{smooth1}
\end{align}

\noi
From 
\eqref{OP1a}, we have
\begin{align}
\coth (\dl n) - \sgn (n) = \sgn(n) \cdot \frac 2{e^{2|\dl n|} - 1}
\label{smooth2}
\end{align}

\noi
for $n \ne 0$.
Then, 
by noting that $x^{s+1} \le C_s( e^{2x} - 1)$ for any $x \ge  0$, 
it follows from~\eqref{smooth1} and~\eqref{smooth2} that 
\begin{align*}
\sup_{n \in \Z \setminus \{0\}}
\jb{n}^s  |\ft \Qdl(n)|
\les \dl^{-1}(1+ \dl^{-s}), 
\end{align*}

\noi
from which \eqref{smooth0} follows.
\end{proof}

We now recall the following fractional Leibniz rule in the periodic setting;
see \cite{BOZ}.
See also 
 \cite[Lemma~9.A.2]{IK2} and \cite[Lemma~3.4]{GKO}.

\begin{lemma} \label{LEM:prod1}
Let $s > 0$ and   $1 <  p_j,q_j,r \le \infty$, $j=1,2$, such that $\frac{1}{r}=\frac{1}{p_j}+\frac{1}{q_j}$.
Then, we have 
\begin{align}
\|J^s(fg)\|_{L^r(\T)}\lesssim\| J^s f\|_{L^{p_1}(\T)} \|g\|_{L^{q_1}(\T)}+ \|f\|_{L^{p_2}(\T)} 
\|  J^s g\|_{L^{q_2}(\T)}, 
\label{prod1}
\end{align}

\noi
where $J^s = \jb{\dx}^s$ denotes the Bessel potential of order $-s$.

\end{lemma}

Note that the estimate \eqref{prod1} holds for $\frac 12 \le r \le 1$
as well 
(with additional assumptions such as  $s > \frac 1r -1$ or $s \in 2\N$);
see \cite{BOZ}.

\medskip

We conclude this section by recalling the basic property 
of the translation operator $\tau$ defined by 
\begin{align}
\tau_h(u) (t, x) = \tau (h, u) (t, x) := u(t, x+ ht).
\label{trans1}
\end{align}

\begin{lemma}\label{LEM:trans}

Let $T > 0$ and $s\ge 0$.
Then, $\tau$ defined in \eqref{trans1}
is continuous from 
$\R \times C([0, T]; H^s(\T))$
into $C([0, T]; H^s(\T))$, 
but is not uniformly continuous
in every neighborhood of the origin $(0, 0)$.
For fixed $h \in \R$, 
$\tau_h$ is an isometry on
$C([0, T]; H^s(\T))$
(in particular, Lipschitz continuous).

\end{lemma}

See 
\cite[Propositions 3.2.1 and 3.2.2]{Herr}
for the first two claims.
The second claim is clear from 
$\| \tau_h u(t) - \tau_h v(t)\|_{H^s_x}
=\|  u(t) -  v(t)\|_{H^s_x}$
for any $t \in \R$.

\section{Gauge transform and nonlinear estimates}
\label{SEC:gauge}

In this section, by viewing ILW \eqref{ILW1} as a perturbation 
of BO \eqref{BO1}, 
we first apply the gauge transform for BO
to ILW.
We then follow the argument in \cite{MP}
and establish a priori estimates.

\subsection{Gauge transform}
By applying the Galilean transform 
\begin{align}
v(t, x) = \tau_{\dl^{-1}}(u)(t, x) = u(t, x + \dl^{-1}t),
\label{gauge0} 
\end{align}

\noi
where $\tau_{\dl^{-1}} = \tau(\dl^{-1}, \,\cdot\,)$ is as in \eqref{trans1}, 
we can rewrite \eqref{ILW1}
as the following renormalized ILW:
\begin{align}
\dt v -
 \Tdl \dx^2 v =  \dx (v^2)    
\label{ILW2}
\end{align}

\noi
with $v|_{t = 0} = u_0$, 
where $\Tdl$ 
is as in \eqref{OP0a}.

Following the recent work \cite{IS}, 
our main strategy 
is to view 
\eqref{ILW2} as a perturbation of the BO equation \eqref{BO1} as follows:
\begin{align}
\dt v - \H \dx^2 v = \Qdl\dx v +  \dx (v^2), 
\label{ILW3}
\end{align}

\noi
where $\Qdl$ is as in \eqref{OP3}, 
and
apply the gauge transform for the BO equation \eqref{BO1}
(see \eqref{gauge4} below)
to \eqref{ILW3}. 
Lemma \ref{LEM:smooth}
shows that 
$\Qdl$ has a smoothing property of any order,
 which  allows
us to indeed view the first term 
on the right-hand side of \eqref{ILW3} 
as a  perturbation
in a suitable sense.

In a seminal work \cite{TAO04}, 
Tao introduced a gauge transform, analogous
to the Hopf-Cole transform for the heat equation, 
in studying low regularity well-posedness of BO \eqref{BO1}.
Let us now recall
the BO gauge transform in the periodic setting  \cite{Moli2, MP}
and apply it to \eqref{ILW3}.
Let $F = F[v]$ be the mean-zero primitive of $v$ defined by 
\begin{align}
F = F[v] =  \dx^{-1} \P_{\ne 0} v, 
\label{gauge1}
\end{align}

\noi
where 
$\dx^{-1}$ denotes
the Fourier multiplier operator with 
symbol $\ind_{\{n\ne 0\}} \cdot (in)^{-1}$
and 
$\P_{\ne 0}$ denotes the
projection onto the non-zero Fourier modes.
In particular, 
for a mean-zero function $v$, we have
\begin{align}
v = \P_{\ne 0} v =  \dx F.
\label{gauge2}
\end{align}

\noi
By noting
$\dx^{-1} \dx 
= \P_{\ne 0}$, 
it is easy to see that, 
if $v$ is a solution to \eqref{ILW3}, 
then $F = F[v]$ satisfies the following equation:
\begin{align}
\dt F - \H \dx^2 F = \Qdl v + v^2 - \P_0(v^2), 
\label{ILW4}
\end{align}

\noi
where $\P_0 = \Id - \P_{\ne 0}$
denotes the projection onto the zeroth Fourier coefficient.

\begin{remark}\label{REM:diff}\rm
For $j = 1, 2$, let  $v_j$ be a function  $C([0, T]; L^2(\T))$
such that $\ft v_j(t, 0) = 0$ for any $t \in [0, T]$.
Then, by letting $F_j$ be the mean-zero primitive of $v_j$
defined by $F_j = \dx^{-1} \P_{\ne 0} v_j$, 
it follows from 
Bernstein's inequality that 
\begin{align}
\| F_1(0) - F_2(0)\|_{L^\infty} \les \| v_1 (0) - v_2(0)\|_{L^2}
\label{diff1}
\end{align}

\noi
and 
\begin{align}
\| F_1- F_2\|_{L^\infty_{T, x}} \les \| v_1 - v_2\|_{L^\infty_T L^2_x}.
\label{diff2}
\end{align}

\noi
These bounds allow us to prove local Lipschitz 
continuity of the solution map to \eqref{ILW2}
in $H^s_0(\T)$.
We point out that the bounds \eqref{diff1} and \eqref{diff2}
do not hold on the real line 
due to the low-frequency issue.
In order to overcome this difficulty in the real line case, 
an extra assumption on the low-frequency part of initial data
was imposed in 
Section 4 of \cite{MP}, 
namely, 
\begin{align}
\P_\LO v_1(0) = \P_\LO v_2(0), 
\label{diff3}
\end{align}

\noi
See 
\cite[Lemma 4.1]{MP}, 
where the bounds \eqref{diff1} and \eqref{diff2}
are established in  the real line case under the extra assumption
\eqref{diff3}.

\end{remark}

We now introduce the gauge transform (for the periodic BO equation):\footnote{In the current periodic setting,
it is possible to define a gauge transform with $\P_+$ instead of $\P_{+,  \hi}$
(as in \cite{Moli1, Moli2})
and carry out the analysis presented in this paper.
We, however, keep the definition of the gauge transform \eqref{gauge4}
with $\P_{+,  \hi}$ so that it is easier to compare the discussion in this paper
with the presentation in \cite{MP}.
}
\begin{align}
W = \P_{+,  \hi}(e^{i F}).
\label{gauge4}
\end{align}

\noi
By noting 
\begin{align*}
\Id = \P_{+, \hi} + \P_{\lo} + \P_{-, \hi}
\qquad \text{and}
\qquad 
 \P_{+, \hi}(\P_{-, \hi} f \cdot \P_- g)  = 0, 
\end{align*}

\noi
 a direction computation with 
\eqref{gauge4},  \eqref{ILW4},  and \eqref{gauge2}
 yields
\begin{align}
\begin{split}
\dt W - \H \dx^2 W 
& = \dt W + i \dx^2 W\\
& = i \P_{+, \hi}\big(e^{iF} (\dt F + i \dx^2F - (\dx F)^2)\big)\\
& = - 2  \P_{+, \hi}(W \P_- \dx v)
- 2  \P_{+, \hi}(\P_\lo e^{iF} \P_- \dx v) \\
& \quad 
+ i \P_{+, \hi}(e^{iF}  \Qdl v) - i  \P_0( v^2) W. 
\end{split}
\label{ILW5}
\end{align}

%
%
\noi
Finally, we define $w = w(v)$ by setting
\begin{align}
w = \dx W = 
\dx \P_{+,  \hi}(e^{i F})
= 
i \P_{+,  \hi}(e^{i F} v).
\label{gauge6}
\end{align}

\noi
By differentiating \eqref{ILW5}, 
we then obtain
\begin{align}
\begin{split}
\dt w - \H \dx^2 w 
& = - 2 \dx \P_{+, \hi}(W \P_- \dx v)
- 2 \dx \P_{+, \hi}(\P_\lo e^{iF} \P_- \dx v) \\
& \quad 
+ i \dx\P_{+, \hi}(e^{iF}  \Qdl v) - i  \P_0( v^2) w\\
& =: \NN_\dl(w, v).
\end{split}
\label{ILW6}
\end{align}

\noi
As compared to \cite[(3-4)]{MP} (modulo
constants), 
we have two extra terms 
on the right-hand side of \eqref{ILW6}.
The third term appears from the difference between
the ILW and BO equations, 
whereas the fourth term appears due to the current periodic setting.

Let us conclude this discussion by 
expressing  $v$ 
in terms of  the gauged function~$w = w(v)$.
From \eqref{gauge2} and \eqref{gauge6}
with 
we can write $v$ as
\begin{align*}
\begin{split}
v & = \dx F = e^{-i F} \cdot e^{iF} \dx F
= - i e^{-i F} \cdot \dx e^{iF}\\
& = - i e^{-i F}  w  
+  e^{-i F} \P_\lo (e^{iF}  v)
-i   e^{-i F}\dx \P_{-, \hi}  e^{iF} .
\end{split}
\end{align*}

\noi
By applying $\P_{+, \HI}$, we then obtain
\begin{align}
\begin{split}
\P_{+, \HI} v 
& = - i \P_{+, \HI}(e^{-i F}  w  )
+  \P_{+, \HI}\big(\P_{+, \hi} e^{-i F} \cdot \P_\lo (e^{iF}  v)\big)\\
& \quad 
-i  \P_{+, \HI}( \P_{+, \HI}e^{-i F} \cdot \dx \P_{-, \hi} e^{iF}), 
\end{split}
\label{ILW8}
\end{align}

\noi
where we have used
$\P_{+, \HI}\big((\Id - \P_{+, \hi})f\cdot  \P_\lo g\big) 
= \P_{+, \HI}\big( (\Id - \P_{+, \HI}) f\cdot \P_{-, \hi} f\big)
= 0$.

\subsection{Product estimates}

We first recall the following product estimate;
see \cite[Lemma~3.2]{Moli1}.

\begin{lemma}\label{LEM:prod3}
Let $s, s_1, s_2  \ge 0$ and $1 < q, q_1, q_2 < \infty$
such that $s_1 \ge s$, $1+s = s_1 + s_2$, and 
$\frac 1q = \frac 1{q_1} + \frac 1{q_2}$.
Then, we have 
\begin{align*}
\| D^s \P_+(f \P_- \dx g) \|_{L^q}
\les \| D^{s_1} f\|_{L^{q_1}} \| D^{s_2} g \|_{L^{q_2}}.
\end{align*}

\end{lemma}

Next, we
recall the following lemma 
on  the multiplication with  $e^{\pm iF}$ in Sobolev spaces;
see
 \cite[Lemma 2.7]{MP}.
 See also  \cite[Lemma 3.1]{Moli1}.

\begin{lemma} \label{LEM:prod2}
Let $2 \le q <\infty$
 and $0 \le s \le \frac 12$.
For $j = 1, 2$, 
let $F_j$ denote the mean-zero primitive of
a mean-zero real-valued function $f_j$ on $\T$
such that $\dx F_j = f_j$.
Then, we have 
\begin{align*}
\|J^s (e^{\pm i F_j} g)\|_{L^q}
\les 
\big(1+\|f_j\|_{L^2}\big)\|J^{s}g\|_{L^q} 
\end{align*}

\noi
and
\begin{align*}
& \big\|J^s \big((e^{\pm i F_1} - e^{\pm i F_2}) g\big)\big\|_{L^q}\\
& \quad
\les 
\Big(\|f_1 - f_2\|_{L^2}
+ \|e^{\pm i F_1} - e^{\pm i F_2}\|_{L^\infty}
\big(1+\|f_1\|_{L^2}\big)\Big)
\|J^{s}g\|_{L^q}.
\end{align*}

\end{lemma}

In \cite[Lemma 2.7]{MP}, 
Lemma \ref{LEM:prod2} was proven  for $0 \le s \le \frac 1q$
but their proof easily extends to the range $0 \le s \le \frac 12$.
In \cite[the last line on p.\,372]{MP}, 
the condition $s = \frac 1q - \frac 1{q_1}$
was used to control 
$\| g \|_{L^{q_1}}$ by 
$\| J^s g \|_{L^{q}}$, 
which yielded the restriction $s \le \frac 1q$. 
We can, however, use Sobolev's inequality
with 
$s \ge  \frac 1q - \frac 1{q_1}$; see~\cite{BO}.
In repeating the proof of \cite[Lemma 2.7]{MP}, 
the restriction $s \le \frac 12$
appears in 
the bound 
$\| D^{s + \frac 12} \P_\hi e^{\pm i F_j}\|_{L^2}
\les
\| \dx e^{\pm i F_j}\|_{L^2}$
in \cite[p.\,373]{MP}.

\subsection{Nonlinear estimates}

We now state nonlinear estimates
which allow us to establish a priori bounds
on a solution $v$ to \eqref{ILW2}
and the associated gauged function $w$ in \eqref{gauge6}.
The following results follow
from slight modifications 
of the corresponding results for the $\dl = \infty$ case;
see Propositions 3.2 and 3.4 in \cite{MP}.
The first lemma controls
a solution $v$ to \eqref{ILW2}
in terms of the gauged function $w$ defined in \eqref{gauge6}.

\begin{lemma}\label{LEM:MP1}
Let $0\le s \le 1$ and $0 \le \kk \le 1$.
Then, there exists $\ta > 0$ such that 
\begin{align}
\| v \|_{X^{s-\kk, \kk}(T)}
 \les T^\ta (1 + \dl^{-1}) \|v\|_{L^\infty_T H^s_x}
+  \|v\|_{L^4_{T, x}} \|J^s v\|_{L^4_{T, x}}
\label{MP1a}
\end{align}

\noi
for
any $0 < \dl \le \infty$, $0 < T \le 1$, and 
any smooth solution $v$ to \eqref{ILW2} on the time interval  $[0, T]$.
Moreover if $0 \le s \le \frac 12$, we have
\begin{align}
\begin{split}
\| J^s v \|_{L^p_T L^q_x}  
& \les
\| v(0) \|_{L^2}
+ \big(1 + 
 \|v\|_{L^\infty_T L^2_x}\big)
\|w\|_{Y^{s, \frac 12}(T)} \\
& \quad + \big(1 + 
 \|v\|_{L^\infty_T L^2_x}\big)
 \|v \|_{L^\infty_T L^2_x}
  \|v \|_{L^\infty_T H^s_x}
+  
T\dl^{-1}
 \|v \|_{L^\infty_T L^2_x}
\end{split}
\label{MP1b}
\end{align}

\noi
for $(p, q) = (\infty, 2)$ or $(4, 4)$, 
where $w = w(v)$ is as in \eqref{gauge6}.
%

\end{lemma}

\begin{proof}
Proceeding as in the proof  of  Proposition 3.2 in \cite{MP}, 
we have 
\begin{align}
\| \wt v\|_{X^{s - \kk, \kk}} \les \| \dt v -  \H \dx^2 v\|_{L^2_T H^{s-1}_x}  
+ T^\ta \|v\|_{L^\infty_T H^s_x}
\label{MP1c}
\end{align}

\noi
for any $0 \le \kk \le 1$, 
where $\wt v(t) = \eta(t) v(t)$ is an extension 
of $v$ from $[0, T]$ to $\R$.
By substituting~\eqref{ILW3} to the first term on the right-hand side of \eqref{MP1c}
and applying Lemma \ref{LEM:smooth}
and 
the fractional Leibniz rule
(Lemma \ref{LEM:prod1}), we have
\begin{align}
\begin{split}
\| \dt v -  \H \dx^2 v\|_{L^2_T H^{s-1}_x}  
& \les \|\Qdl v\|_{L^2_T H^s_x}
+ \|v\|_{L^4_{T, x}} \|J^s v\|_{L^4_{T, x}}\\
& \les
T^\ta \dl^{-1}  \| v\|_{L^\infty_T H^s_x}
+  \|v\|_{L^4_{T, x}} \|J^s v\|_{L^4_{T, x}}.
\end{split}
\label{MP1d}
\end{align}

\noi
Hence, \eqref{MP1a} follows from \eqref{MP1c} and \eqref{MP1d}.

Note that
 the identity 
 \eqref{ILW8}
does not involve the depth parameter $\dl$ (or an equation)
and is the same as (3-6) in \cite{MP} (modulo constants).
Thus, provided that (i)~$(p, q) = (\infty, 2)$ 
and $0 \le s \le \frac 12$
or (ii)~$(p, q) = (4, 4)$
and $0 \le s \le \frac 14$, 
we can proceed as in  the proof of Proposition 3.2 in~\cite{MP}.
Note, however, that in~\cite{MP}, the low-frequency contribution
was estimated by studying 
the Duhamel formulation of the BO equation; see \cite[(3-17) and (3-18)]{MP}.
In our case, we instead need to
study the Duhamel formulation of \eqref{ILW3}, 
where 
the additional contribution as compared to \cite{MP} is estimated 
via  Lemma \ref{LEM:smooth} as
\begin{align}
\bigg\| J^s \int_0^t S(t-t') \P_{\LO} \Qdl \dx v (t') dt' \bigg\|_{L^p_T L^q_x}  
\les T\dl^{-1} \|v \|_{L^\infty_T L^2_x}.
\label{MP1dd}
\end{align}

\noi
This yields 
the second estimate~\eqref{MP1b}
except for the case
 $(p, q) = (4, 4)$
and $\frac 14 <  s \le \frac 12$.

It remains to consider the case
 $(p, q) = (4, 4)$
and $\frac 14 <  s \le \frac 12$.
In view of \eqref{ILW8}, 
we write $v = \P_+ v + \cj {\P_+ v}$ (recall that $v$ is real-valued) and 
\begin{align*}
\P_+ v & = 
\P_{+, \LO} v 
 - i \P_{+, \HI}(e^{-i F}  w  )
+  \P_{+, \HI}\big(\P_{+, \hi} e^{-i F} \cdot \P_\lo (e^{iF}  v)\big)\\
& \quad 
-i  \P_{+, \HI}( \P_{+, \HI}e^{-i F} \cdot \dx \P_{-, \hi} e^{iF})\\
& =: A_1 + A_2 + A_3 + A_4.
\end{align*}

\noi
Note that 
 the estimates
(3-18)  (see also \eqref{MP1dd} above) and (3-14) in \cite{MP}
on  $A_1$ and  $A_3$, respectively, 
also hold for $\frac 14 <  s \le \frac 12$.
As for $A_2$, by applying  Lemmas \ref{LEM:prod2}
and  \ref{LEM:L4}, we have 
\begin{align}
\begin{split}
\| J^s A_2\|_{L^4_{T, x}}
& = \| J^s \P_{+, \HI}(e^{-i F}  w  )\|_{L^4_{T, x}}
\les
\big( 1+ \|v\|_{L^\infty_T L^2_x} \big)
 \| J^s   w  \|_{L^4_{T, x}}\\
& \les
\big( 1+ \|v\|_{L^\infty_T L^2_x} \big)
 \|  w  \|_{Y^{s, \frac 12}(T)}
 \end{split}
 \label{MP1e}
\end{align}

\noi
extending  \cite[(3-13)]{MP} to $\frac 14 < s \le \frac 12$.

Next, we treat $A_4$.
In view of the frequency localizations, we have
\begin{align}
\begin{split}
A_4 = 
- i \sum_{j, j_1 = 3}^\infty
\sum_{j_2 = 1}^\infty
\ind_{j_1 \ge j,   j_2}\cdot 
  \Q_j
  \P_{+}( \Q_{j_1} \P_{+}e^{-i F} \cdot \Q_{j_2} \dx \P_{-} e^{iF}).
\end{split}
\label{MP4a}
\end{align}

\noi
In the following, we  work for fixed time $t$.
For simplicity of notations, 
we  suppress $t$-dependence.
By 
the Littlewood-Paley theorem and  Minkowski's integral inequality, we have 
\begin{align}
\begin{split}
  \| J^s A_4\|_{L^4_x}
&  \sim  \bigg\|\bigg(
\sum_{j = 3}^\infty
| \Q_j J^s \P_{+}( \P_{+, \HI}e^{-i F} \cdot \dx \P_{-, \hi} e^{iF})|^2\bigg)^\frac{1}{2}\bigg\|_{L^4_x}\\
&  \les  
\bigg(
\sum_{j = 3}^\infty2^{2(s-\frac 14) j}
\| \Q_j J^\frac 14 \P_{+}( \P_{+, \HI}e^{-i F} \cdot \dx \P_{-, \hi} e^{iF})\|_{L^4_x }^2\bigg)^\frac 12.
\end{split}
\label{MP4b}
\end{align}

\noi
We separately estimate the contributions to \eqref{MP4b}
by considering the following two cases:
$2^{j_1} \sim 2^j$
and  $2^{j_1} \sim 2^{j_2} \gg 2^j$, 
where $j_1$ and $j_2$ are as in \eqref{MP4a}.

\medskip

\noi
$\bullet$
{\bf Case 1:} $2^{j_1} \sim 2^j$.
\\
\indent
In this case, we have $|j_1 - j|\le 2$.
Then, proceeding as in 
\cite[(3-15)]{MP} with Lemma \ref{LEM:prod3}, 
Sobolev's inequality, and \eqref{gauge2}, we have 
\begin{align}
\begin{split}
\eqref{MP4b} 
& \les
\sum_{k = -2}^2
\bigg(
\sum_{j_1 = 3}^\infty2^{2(s-\frac 14) j_1}
\| \Q_{j_1+k} D^\frac 14 \P_{+}( \Q_{j_1} \P_{+}e^{-i F} \cdot \dx \P_{-, \hi} e^{iF})\|_{L^4_x }^2\bigg)^\frac 12
\\
& \les
\bigg(
\sum_{j_1 = 3}^\infty2^{2(s-\frac 14) j_1}
\|  D^\frac 58 \Q_{j_1} \P_{+}e^{-i F}\|_{L^8_x}^2
\bigg)^\frac 12
\|D^\frac 58 \P_{-, \hi} e^{iF}\|_{L^8_x }\\
& \les
\bigg(
\sum_{j_1 = 3}^\infty2^{2(s-\frac 14) j_1}
\|   \Q_{j_1}\dx \P_{+}e^{-i F}\|_{L^2_x}^2
\bigg)^\frac 12
\|\dx  \P_{-, \hi} e^{iF}\|_{L^2_x }\\
& \les 
\bigg(
\sum_{j_1 = 3}^\infty2^{2(s-\frac 14) j_1}
\|   \Q_{j_1} \P_{+}(e^{-i F}v)\|_{L^2_x}^2
\bigg)^\frac 12
\|v\|_{L^2_x }.
\end{split}
\label{MP4c}
\end{align}

\noi
By Lemma \ref{LEM:prod2}
(recall that $\frac 14 <  s \le \frac 12$), we have 
\begin{align}
\begin{split}
& \bigg(
\sum_{j_1 = 3}^\infty2^{2(s-\frac 14) j_1}
\|   \Q_{j_1} \P_{+}(e^{-i F}v)\|_{L^2_x}^2
\bigg)^\frac 12\\ 
& \hphantom{XXXXX}
 \les \|e^{-i F}v\|_{H^{s- \frac 14}_x}
\les \big(1 + \|v\|_{L^2_x}\big)\|v\|_{H^{s- \frac 14}_x}.
\end{split}
\label{MP4d}
\end{align}

\noi
Hence, from \eqref{MP4b}, \eqref{MP4c}, and \eqref{MP4d},  
we conclude that the contribution from this case 
(after taking the $L^4_T$-norm) is
bounded by the right-hand side of \eqref{MP1b}.

\medskip

\noi
$\bullet$
{\bf Case 2:} $2^{j_1} \sim 2^{j_2} \gg 2^j$.
\\
\indent
In this case, we have $|j_1 - j_2|\le 2$ and $j \le  j_1  -3$.
Proceeding as in \eqref{MP4c}, we have 
\begin{align}
\begin{split}
\eqref{MP4b} 
& \les
\sum_{k = -2}^2
\sum_{j_1 = 3}^\infty
\bigg(
\sum_{j = 3}^{j_1-3} 2^{2(s-\frac 14) j}\\
& \hphantom{XXXXX}
\times
\|  D^\frac 14 \P_{+, \HI}( \Q_{j_1} \P_{+}e^{-i F} \cdot  \Q_{j_1+k}\dx \P_{-, \hi} e^{iF})\|_{L^4_x }^2\bigg)^\frac 12
\\
& \les
\sum_{k = -2}^2 
\sum_{j_1 = 3}^\infty
2^{(s-\frac 14) j_1}
\|  D^\frac 58 \Q_{j_1} \P_{+}e^{-i F}\|_{L^8_x}
\| \Q_{j_1+k} D^\frac 58 \P_{-, \hi} e^{iF}\|_{L^8_x }\\
& \les
\sum_{k = -2}^2 
\sum_{j_1 = 3}^\infty
2^{(s-\frac 14) j_1}
\|  \Q_{j_1} \dx \P_{+}e^{-i F}\|_{L^2_x}
\|\Q_{j_1+k} \dx  \P_{-, \hi} e^{iF}\|_{L^2_x }\\
& \les 
\bigg(
\sum_{j_1 = 3}^\infty2^{2(s-\frac 14) j_1}
\|   \Q_{j_1} \P_{+}(e^{-i F}v)\|_{L^2_x}^2
\bigg)^\frac 12
\|v\|_{L^2_x }, 
\end{split}
\label{MP4e}
\end{align}

\noi
where we used Cauchy-Schwarz's inequality (in $j_1$) in the last step.
\noi
Hence, from \eqref{MP4b}, \eqref{MP4e}, and \eqref{MP4d},  
we conclude that the contribution from this case 
(after taking the $L^4_T$-norm) 
is
bounded by the right-hand side of \eqref{MP1b}.

\smallskip

Putting everything together, 
we  conclude \eqref{MP1b}
when $(p, q) = (4, 4)$
and $\frac 14 <  s \le \frac 12$.
\end{proof}

The following proposition allows us to  estimate the gauged function $w$
in terms of $v$.

\begin{proposition}\label{PROP:MP2}
Let $0 \le s \le \frac 12$.
Then, there exists $\ta > 0$ such that 
\begin{align}
\begin{split}
& \|\ind_{[0, T]}\cdot \NN_{\dl}(w, v)\|_{Y^{s, -\frac 12}} \\
& \quad \les 
 \|v \|_{L^4_{T, x}}^2
 +  \Big( T^\ta \| v \|_{L^\infty_T L^2_x} + \|  v \|_{L^4_{T, x}} 
+ \| v \|_{X^{-1, 1}(T)}\Big)\|w\|_{X^{s, \frac 12}(T)}\\
& \qquad 
+ T^\ta \|v \|_{L^\infty_TL^2_x}^2 
\| w \|_{Y^{s, \frac 12}(T)}
+ T^\ta
 \dl^{-1}(1+\dl^{- 1}) \big( 1 + \| v\|_{L^\infty_T L^2_x}\big)^2
\|  v \|_{L^{\infty}_TH^s_x}
\end{split}
\label{MP2a}
\end{align}

\noi
for
any $0 < \dl \le \infty$, $0 < T \le 1$, 
and 
any smooth solution $v$ to \eqref{ILW2}  the time interval on $[0, T]$, 
where $w = w(v)$ is as in \eqref{gauge6} and 
 $\NN_{\dl}(w, v)$ is as in \eqref{ILW6}.

\end{proposition}

Note that  the last term on the right-hand side of \eqref{MP2a}
contains a linear term $\dl^{-1} (1+\dl^{-1}) \|  v \|_{L^{\infty}_TH^s_x}$.
In order to implement a bootstrap argument, 
we need additional smallness.
This
is the primary reason that we keep a small power of $T$ for this term.
The factor $T^\ta$ on the other terms are used in \eqref{WP3}.
See Section \ref{SEC:LWP}.

We now present the
proof of Proposition \ref{PROP:MP2}.

\begin{proof}[Proof of Proposition \ref{PROP:MP2}]
By arguing as in  \cite[Proposition~3.5 and Lemma~3.7]{MP}, 
we can bound the contribution from 
 the first two terms on the right-hand side of \eqref{ILW6}
by the right-hand side of \eqref{MP2a}.

Let us first treat the fourth
term on the right-hand side of \eqref{ILW6}.
By \eqref{loc2}, we have 
\begin{align}
\| \ind_{[0, T]} (t)\cdot  \P_0( v^2)  w\|_{L^2_t H^s_x}
  \les T^\ta \|v \|_{L^\infty_TL^2_x}^2 
\| w \|_{L^\infty_T H^s_x}
 \les T^\ta \|v \|_{L^\infty_TL^2_x}^2 
\| w \|_{Y^{s, \frac 12}(T)}.
\label{MP2b}
\end{align}

\noi
We now turn our attention to the third term in \eqref{ILW6}.
Proceeding as in \cite[(3-47)]{MP}
with $\dx F =  v$
and Lemma \ref{LEM:prod2}, we have
\begin{align}
\begin{split}
&
   \|\ind_{[0, T]} (t)\cdot  \dx\P_{+, \hi}(e^{i   F}  \Qdl  v) \|_{Y^{s, -\frac 12}}
 \les 
\|   \dx\P_{+, \hi}(e^{i   F}  \Qdl  v) \|_{L^{1+\kk}_TH^s_x}\\
& \quad \le 
\|  \P_{+, \hi}(e^{i   F}  v  \Qdl  v) \|_{L^{1+\kk}_TH^s_x}
+ \|    \P_{+, \hi}(e^{i   F}  \Qdl \dx  v) \|_{L^{1+\kk}_TH^s_x}\\
& \quad 
\les T^\ta  \big( 1 + \| v\|_{L^\infty_T L^2_x}\big)
\Big(
\| v  \Qdl  v \|_{L^{\infty}_TH^s_x}
+ \| \Qdl \dx  v \|_{L^{\infty}_TH^s_x}
\Big)
\end{split}
\label{MP2c}
\end{align}

\noi
for any  $\kk > 0$.
From the fractional Leibniz rule (Lemma \ref{LEM:prod1}), 
Sobolev's embedding, 
and Lemma \ref{LEM:smooth}, 
we have 
\begin{align}
\begin{split}
\| v  \Qdl  v \|_{L^{\infty}_TH^s_x}
& \les 
\| v \|_{L^{\infty}_TH^s_x}\|  \Qdl  v \|_{L^{\infty}_{T,  x}}
+ 
\| v \|_{L^{\infty}_TL^2_x}\|  \Qdl  v \|_{L^\infty_T W^{s, \infty}_x}\\
& \les 
\| v \|_{L^{\infty}_TH^s_x}\|  \Qdl  v \|_{L^{\infty}_{T}H^1_x}
+ 
\| v \|_{L^{\infty}_TL^2_x}\|  \Qdl  v \|_{L^\infty_T H^{s+1}_x}\\
& \les \dl^{-1}(1+\dl^{- 1})
\| v \|_{L^{\infty}_TH^s_x}\| v \|_{L^{\infty}_TL^2_x}.
\end{split}
\label{MP2d}
\end{align}

\noi
By Lemma \ref{LEM:smooth}, we also have
\begin{align}
 \| \Qdl \dx  v \|_{L^{\infty}_TH^s_x}
\les 
 \dl^{-1}(1+\dl^{- 1})
\|  v \|_{L^{\infty}_TH^s_x}.
\label{MP2e}
\end{align}

\noi
Hence, from  \eqref{MP2c}, \eqref{MP2d}, and \eqref{MP2e}, 
we obtain
\begin{align}
\begin{split}
   \| \ind_{[0, T]} (t)\cdot\dx\P_{+, \hi}(e^{i   F}  \Qdl  v) \|_{Y^{s, -\frac 12}}
 \les 
T^\ta\dl^{-1}(1+\dl^{- 1}) \big( 1 + \| v\|_{L^\infty_T L^2_x}\big)^2
\|  v \|_{L^{\infty}_TH^s_x}.
\end{split}
\label{MP2f}
\end{align}

Therefore, from \eqref{MP2b} and \eqref{MP2f}, 
we conclude \eqref{MP2a}.
\end{proof}

\begin{remark}\label{REM:MP} \rm
In the proof of Proposition \ref{PROP:MP2}, 
we handled the first two terms on the right-hand side of \eqref{ILW6}
by simply invoking 
the bilinear estimates in  \cite[Proposition 3.5 and Lemma~3.7]{MP}.
We point out  that 
these bilinear estimates (which hold  on both the real line and  the circle)
 are the key ingredients
 for the simplification and unification
of the $L^2$ well-posedness argument for the BO equation \eqref{BO1}
 presented in \cite{MP}.
\end{remark}

\begin{remark} \label{REM:conti} \rm

We point out that  the map $T\mapsto 
 \|\ind_{[0, T]}\cdot \NN_{\dl}(w, v)\|_{Y^{s, -\frac 12 }}$
 is continuous
 for smooth $w$ and $v$.
 Indeed, given $h > 0$, it follows from the triangle inequality that 
\begin{align*}
  \Big|\| & \ind_{[0, T+ h]}\cdot \NN_{\dl}(w, v)\|_{Y^{s, -\frac 12 }}
-   \|\ind_{[0, T]}\cdot \NN_{\dl}(w, v)\|_{Y^{s, -\frac 12 }}\Big|\\
& \le
 \|\ind_{[T, T+h]}\cdot \NN_{\dl}(w, v)\|_{Y^{s, -\frac 12 }}\\
& \les   \|\ind_{[T, T+h]}\cdot \NN_{\dl}(w, v)\|_{X^{s, 0}}
 = 
 \|\NN_{\dl}(w, v)\|_{L^2([T, T+h]; H^s)}\\
 & \too 0, 
\end{align*}
 
 \noi
 as $h \to 0$.
 A similar computation holds for $ h < 0$, tending to $0$.
 See 
 \cite[Lemma A.8]{BOP2}, 
 \cite[Lemma 4.4]{Bring}, 
 and \cite[Remark 3.3]{OQS}
for the continuity property (in $T$) of the $X^{s, b}$-norm and its variants
under the multiplication by 
$ \ind_{[0, T]}$.
See also 
\cite[Lemma 6.3]{KP0}
and \cite[Lemma 8.1]{GO}
for similar statements
in the context of the more complicated short-time Fourier restriction norm method.
%
%

\end{remark}

\section{Well-posedness of ILW in $L^2$}
\label{SEC:LWP}

In this section, we present the proof of 
Theorem \ref{THM:1}.
With Lemma~\ref{LEM:MP1} and Proposition~\ref{PROP:MP2}
in hand,  it essentially follows from the discussion in  Section 4 of \cite{MP}, 
and thus 
we keep our presentation brief.
Our main goal is to prove local well-posedness
of the renormalized ILW~\eqref{ILW2} in $L^2(\T)$.
Once we establish local well-posedness, 
global well-posedness follows from the $L^2$-conservation.

Fix $0 < \dl < \infty$ and $0 \le s \le\frac 12$.
We only discuss the case where  initial data has sufficiently small $L^2$-norm.
More precisely, 
in the remaining part of this section, 
we assume that all the initial data in the following discussion have the $L^2$-norm less than 
some sufficiently small 
$\eps_0 = \eps_0(\min(\dl, 1)) > 0$ (to be chosen later).
See Remark \ref{REM:scaling} for the general case.

\subsection{A priori estimates for smooth solutions}
\label{SUBSEC:LWP1}

Let us first derive a priori estimates
for smooth solutions to \eqref{ILW2}.
Let $v$ be a smooth global  solution to~\eqref{ILW2}
and 
$w = w(v)$ be the gauged function defined in~\eqref{gauge6}.
Given $T > 0$, set
\begin{align}
N_T^s(v) = \max \Big( \| v\|_{L^\infty_T H^s_x}, \| v\|_{L^4_T W^{s, 4}_x}, 
\| w(0)\|_{H^s}, 
\| \ind_{[0, T]} \cdot \NN_{\dl}(w, v) \|_{Y^{s, -\frac 12}}\Big), 
\label{WP1}
\end{align}

\noi
where $\NN_{\dl}(w, v)$ is as in \eqref{ILW6}.
We first  note that
from the Duhamel formulation of~\eqref{ILW6}
with  \eqref{loc1} (for $Y^{s, \frac 12}(T)$),
Lemma \ref{LEM:lin},  and \eqref{WP1}, 
we have 
\begin{align}
\| w \|_{Y^{s, \frac 12}(T)}
\les 
\| w(0)\|_{H^s} + 
\|\ind_{[0, T]}\cdot  \NN_{\dl}(w, v) \|_{Y^{s, -\frac 12}}
\les N^s_T(v).
\label{WP3a}
\end{align}

\noi
Thanks to the smoothness of $v$ (and $w$), 
we see that 
$N^s_T(v)$  is continuous in $T$; see Remark~\ref{REM:conti}.
Moreover, from the definition, 
$N^s_T(v)$ is non-decreasing in $T$.
From~\eqref{gauge6} and Lemma~\ref{LEM:prod2}, we have
\begin{align}
\| w(0)\|_{H^s} \les \big(1 + \|v(0)\|_{L^2}\big) \|v(0)\|_{H^s}.
\label{WP2}
\end{align}

\noi
See also \cite[(3-49)]{MP}.
Hence, from \eqref{WP1} and \eqref{WP2}, we have 
\begin{align}
\lim_{T \to 0+} N_T^s(v) 
\les \big(1 + \|v(0)\|_{L^2}\big) \|v(0)\|_{H^s}.
\label{WP3}
\end{align}

\noi
In \eqref{WP3}, 
 the control on 
 $\|  \ind_{[0, T]} \cdot \NN_{\dl}(w, v) \|_{Y^{s, -\frac 12}}$
follows from \eqref{MP2a} and \eqref{MP1a}.
On the other hand, from 
Lemma \ref{LEM:MP1} and Proposition \ref{PROP:MP2}
with \eqref{WP3a} and \eqref{WP2}, 
we have\footnote{Note that the constants in \eqref{WP4} depend
on $s$.  We, however, use \eqref{WP4} only at two regularities ($s$ and $0$)
and thus we can simply take the worse constants in \eqref{WP4}.}
\begin{align}
\begin{split}
N^s_T(v) 
& \le  C_1 \big( 1 + \|v(0)\|_{L^2}\big)\|v(0)\|_{H^s}\\
& \quad 
+ C_2 (1 + \dl^{-1})^2N^0_T(v)\big(1+ N^0_T(v)\big)N^s_T(v)\\
& \quad +  C_3  T^\ta \dl^{-1} ( 1+ \dl^{-1}) N^s_T(v)
\end{split}
\label{WP4}
\end{align}

\noi
for any $0 < T \le 1$.
We point out that
there are linear terms
 on the right-hand sides of~\eqref{MP1b}
and \eqref{MP2a}:
$\|w\|_{Y^{s, \frac 12}(T)}$, 
$T \dl^{-1}
 \|v \|_{L^\infty_T L^2_x}$, 
and $  T^\ta \dl^{-1} ( 1+ \dl^{-1})
\|  v \|_{L^{\infty}_TH^s_x}$, 
where the last  two terms
are bounded by 
 the last term on the right-hand side of~\eqref{WP4}.
As for 
the linear term 
$\|w\|_{Y^{s, \frac 12}(T)}$ on the right-hand side of \eqref{MP1b},  
we can use \eqref{WP3a} and~\eqref{MP2a}
to replace it by $  T^\ta \dl^{-1} ( 1+ \dl^{-1})
\|  v \|_{L^{\infty}_TH^s_x}$ (along with the other terms
on the right-hand side of \eqref{MP2a}), 
which is in turn 
 bounded by 
 the last term on the right-hand side of~\eqref{WP4}.

By choosing sufficiently small  $T_* = T_*(\min(\dl, 1)) > 0$
such that $C_3 T_*^\ta \dl^{-1} ( 1+ \dl^{-1}) \le \frac 12$,
we then have
\begin{align}
\begin{split}
N^s_T(v) 
& \le 2C_1 \big( 1 + \|v(0)\|_{L^2}\big)\|v(0)\|_{H^s}\\
& \quad + 2C_2 (1 + \dl^{-1})^2
N^0_T(v)\big(1+ N^0_T(v)\big)
N^s_T(v)
\end{split}
\label{WP5}
\end{align}

\noi
for any $0 < T \le T_*$.
By a standard continuity argument, 
we have
\begin{align}
N^0_T(v) 
& \les \eps 
\label{WP6}
\end{align}

\noi
for any $0 < T \le T_*$, 
provided that $ \|v(0)\|_{L^2}  = \eps  \le \eps_0$
for some sufficiently small $\eps_0 
= \eps_0(\min(\dl, 1))$.
Then, by applying a continuity argument with \eqref{WP5} and \eqref{WP6}, 
we obtain 
\begin{align}
N^s_T(v) 
& \les \|v(0)\|_{H^s}
\label{WP7}
\end{align}

\noi
for any $0 < T \le T_*$.

%
%

\begin{remark}\rm
From \eqref{WP4},  we can choose 
 $T_* = T_*(\min(\dl, 1)) \ges 1$
and small  $\eps_0 
= \eps_0(\min(\dl, 1)) \ge \eps_1 > 0$, 
 uniformly in $\dl \ge 1$.
The same comment applies to the discussion below.
\end{remark}

\subsection{Difference estimates}
\label{SUBSEC:LWP2}

Let $v_j$, $j = 1, 2$, be smooth global solutions
to \eqref{ILW2}
with 
\begin{align}
\| v_j(0)\|_{L^2} \le \eps \le \eps_0
= \eps_0(\min(\dl, 1)), 
\label{WP7a} 
\end{align}

\noi
and  let $F_j$ be the mean-zero primitive of $v_j$
as defined in \eqref{gauge1}.
Then, as in~\eqref{gauge4}
and~\eqref{gauge6}, 
define the associated gauged functions $W_j$ and $w_j$, $j = 1, 2$, by setting
\begin{align*}
W_j  = \P_{+,  \hi}(e^{i F_j}) \qquad
\text{and}\qquad 
w_j = \dx W_j.
\end{align*}

\noi
Then, from the discussion in the previous subsection (see \eqref{WP6} and \eqref{WP7}
with \eqref{WP3a}), we have 
\begin{align}
\begin{split}
\max_{j = 1, 2}\| w \|_{Y^{0, \frac 12}(T)}
& \les N^0_T(v_j) 
 \les  \eps\le \eps_0, \\
\max_{j = 1, 2}\| w \|_{Y^{s, \frac 12}(T)}
& \les N^s_T(v_j) 
 \les \max_{j = 1, 2}\|v_j(0)\|_{H^s} =:  M 
\end{split}
\label{WP8}
\end{align}

\noi
for any $0 < T\le T_*= T_*(\min(\dl, 1))$,  $j = 1, 2$.

Let $\wt v = v_1 - v_2$, 
$\wt W = W_1 - W_2$, and $\wt w = w_1 - w_2$.
By 
the mean value theorem with \eqref{diff1} and \eqref{diff2}, we have
\begin{align}
\begin{split}
\| e^{\pm iF_1(0)} - e^{\pm iF_2(0)}\|_{L^\infty}
& \leq \| F_1(0)- F_2(0)\|_{L^\infty} \les \| \wt v(0) \|_{ L^2}, \\
\| e^{\pm iF_1} - e^{\pm iF_2}\|_{L^\infty_{T, x}}
& \leq \| F_1- F_2\|_{L^\infty_{T, x}} \les \| \wt v \|_{L^\infty_T L^2_x}.
\end{split}
\label{ILWx5}
\end{align}

\noi
Then, from \eqref{gauge6}, Lemma \ref{LEM:prod2}, and \eqref{ILWx5}
with \eqref{WP7a}, we have 
\begin{align}
\begin{split}
\| \wt w(0)\|_{H^s}
&  \les 
\big(1 + \|v_1(0)\|_{L^2}\big)\|\wt v(0)\|_{H^s}\\
& \quad + 
 \Big( \|\wt v(0) \|_{ L^2}
+ \| e^{iF_1(0)} - e^{iF_2(0)}\|_{L^\infty}\big(1+ \|v_1(0)\|_{ L^2}\big)\Big)
\|v_2(0)\|_{H^s}\\
&  \les 
\|\wt v(0)\|_{H^s}
+ \|v_2(0)\|_{H^s}
\|\wt v(0) \|_{ L^2}.
\end{split}
\label{ZP1}
\end{align}

We now estimate $\wt w$.
From \eqref{ILW6}, we have
\begin{align}
\begin{split}
\dt & \wt w - \H \dx^2 \wt w \\
& = - 2 \dx \P_{+, \hi}(W_1 \P_- \dx \wt v)
- 2 \dx \P_{+, \hi}(\wt W \P_- \dx v_2)\\
& \quad 
- 2 \dx \P_{+, \hi}(\P_\lo e^{iF_1} \P_- \dx \wt v) 
- 2 \dx \P_{+, \hi}\big(\P_\lo (e^{iF_1} - e^{iF_2} )\P_- \dx v_2\big) \\
& \quad 
+ i \dx\P_{+, \hi}(e^{iF_1}  \Qdl \wt v) 
+ i \dx\P_{+, \hi}\big((e^{iF_1} - e^{iF_2})  \Qdl v_2\big) \\
& \quad 
- i  \P_0( v_1^2) \wt w
- i  \P_0( v_1^2 - v_2^2 ) w_2\\
& = : B_1 + \cdots + B_8.
\end{split}
\label{ILWx1}
\end{align}

\noi
Note that the first four terms on the right-hand side
are already treated in Section~4 of~\cite{MP}, 
which, in view of \eqref{WP8} and \eqref{ILWx5}, 
are bounded by  
the first two terms on the right-hand side of \eqref{ZP2} below.\footnote{In fact, there is 
is another term $\eps \|\wt w\|_{X^{s, \frac 12}(T)}$, 
controlling the second term on the right-hand side of \eqref{ILWx1};
see \cite[Proposition 3.5]{MP}.
In \eqref{ZP2}, this term is already hidden on the left-hand side, 
which explains the reason for the factor 2
in front of the last term in \eqref{ZP2}.
}
By a slight modification of \eqref{MP2b}
with \eqref{WP8}, we have 
\begin{align}
\begin{split}
\| B_7 + B_8 \|_{Y^{s, - \frac 12}(T)}
& \les  \|v_1 \|_{L^\infty_TL^2_x}^2 
\| \wt w \|_{Y^{s, \frac 12}(T)}\\
& \quad + 
\big(\|v_1 \|_{L^\infty_TL^2_x}
+ \|v_2 \|_{L^\infty_TL^2_x}\big)
\|\wt v \|_{L^\infty_TL^2_x}
\| w_2 \|_{Y^{s, \frac 12}(T)}\\
& \les \eps^2 
\| \wt w \|_{Y^{s, \frac 12}(T)}
+ 
\eps \| w_2 \|_{Y^{s, \frac 12}(T)}
\|\wt v \|_{L^\infty_TL^2_x}
\end{split}
\label{ILWx2}
\end{align}

\noi
for any $0 < T\le T_*$.
Proceeding as in  the proof of Proposition \ref{PROP:MP2}
with \eqref{WP8}, 
given an absolute  small constant $c_0 > 0$ (to be chosen later), 
by possibly making 
 $T_* = T_*(\min(\dl, 1)) > 0$   smaller, we can show that 
 there exists $\ta > 0$ such that 
\begin{align}
\begin{split}
\| B_5  \|_{Y^{s, - \frac 12}(T)}
& \le  
CT^\ta\dl^{-1}(1+\dl^{- 1}) \big( 1 + \| v_1\|_{L^\infty_T L^2_x}\big)\\
& 
\quad \times 
\Big( \big(1 + \| v_1 \|_{L^{\infty}_TL^2_x}\big)
\|\wt v \|_{L^{\infty}_TH^s_x}
+  \| v_1 \|_{L^{\infty}_TH^s_x}
\|\wt v \|_{L^{\infty}_TL^2_x}\Big)\\
& \le c_0 
\Big(\|\wt v \|_{L^{\infty}_TH^s_x}
+  \| v_1 \|_{L^{\infty}_TH^s_x}
\|\wt v \|_{L^{\infty}_TL^2_x}\Big)
\end{split}
\label{ILWx3}
\end{align}

\noi
for any $0 < T\le T_*$.

Let us now estimate $B_6$.
By Cauchy-Schwarz's inequality, Sobolev's inequality, and Minkowski's
integral inequality  (see \cite[(3-47)]{MP}), we have 
\begin{align}
\begin{split}
\| B_6  \|_{Y^{s, - \frac 12}(T)}
& \les  
\big\|  \dx\P_{+, \hi}\big((e^{iF_1} - e^{iF_2})  \Qdl v_2\big) \big\|_{L^{1+\kk}_TH^s_x}\\
&  \le 
\big\|  \P_{+, \hi}\big(\dx (e^{iF_1} - e^{iF_2})    \Qdl  v_2\big) \big\|_{L^{1+\kk}_TH^s_x}\\
& \quad + \big\|   
 \P_{+, \hi}\big((e^{iF_1} - e^{iF_2}) \Qdl \dx  v_2\big) \big\|_{L^{1+\kk}_TH^s_x}\\
&  \le 
\big\|  \P_{+, \hi}\big(e^{iF_1}(v_1 - v_2)     \Qdl  v_2\big) \big\|_{L^{1+\kk}_TH^s_x}\\
&   \quad 
+ \big\|  \P_{+, \hi}\big((e^{iF_1} - e^{iF_2}) v_2    \Qdl  v_2\big) \big\|_{L^{1+\kk}_TH^s_x}\\
& \quad + \big\|   
 \P_{+, \hi}\big((e^{iF_1} - e^{iF_2}) \Qdl \dx  v_2\big) \big\|_{L^{1+\kk}_TH^s_x}
\end{split}
\label{ILWx4}
\end{align}

\noi
for any $\kk > 0$.
By applying Lemma \ref{LEM:prod2} and 
proceeding as in  \eqref{MP2d} with \eqref{WP8}, we have
\begin{align}
\begin{split}
& \big\|  \P_{+, \hi}\big(e^{iF_1}(v_1 - v_2)     \Qdl  v_2\big) \big\|_{L^{1+\kk}_TH^s_x}\\
& \quad 
\les  
 \dl^{-1}(1+\dl^{- 1})
 \big( 1 + \| v_1\|_{L^\infty_T L^2_x}\big)
\Big( \| v_2 \|_{L^{\infty}_TL^2_x}
\|\wt v \|_{L^{\infty}_TH^s_x}
+  \| v_2 \|_{L^{\infty}_TH^s_x}
\|\wt v \|_{L^{\infty}_TL^2_x}\Big)\\
& \quad 
\les  
 \dl^{-1}(1+ \dl^{-1})
\Big(\eps \|\wt v \|_{L^{\infty}_TH^s_x}
+   \| v_2 \|_{L^{\infty}_TH^s_x}
\|\wt v \|_{L^{\infty}_TL^2_x}\Big).
\end{split}
\label{ILWx4a}
\end{align}

\noi
By Lemma \ref{LEM:prod2}, \eqref{ILWx5}, and  \eqref{MP2d}
with \eqref{WP8}, we have
\begin{align}
\begin{split}
& \big\|  \P_{+, \hi}\big((e^{iF_1} - e^{iF_2}) v_2    \Qdl  v_2\big) \big\|_{L^{1+\kk}_TH^s_x}\\
& \quad 
 \les \Big( \|\wt v \|_{L^\infty_T L^2_x}
+ \| e^{iF_1} - e^{iF_2}\|_{L^\infty_{T, x}}\big(1+ \|v_1\|_{L^\infty_T L^2_x}\big)
\Big)\|v_2 \Qdl  v_2\|_{L^\infty_TH^{s}_x}\\
& \quad 
 \les \dl^{-1}(1+ \dl^{-1})
\big(1+ \|v_1\|_{L^\infty_T L^2_x}\big)\|\wt v \|_{L^\infty_T L^2_x}
\|v_2 \|_{L^\infty_TH^s_x}\|v_2 \|_{L^\infty_TL^2_x}\\
& \quad 
 \les \eps \dl^{-1}(1+ \dl^{-1})
\|v_2 \|_{L^\infty_TH^s_x}
\|\wt v \|_{L^\infty_T L^2_x}.
\end{split}
\label{ILWx6}
\end{align}

\noi
Similarly, 
by Lemma \ref{LEM:prod2}, \eqref{ILWx5}, 
and Lemma \ref{LEM:smooth} with \eqref{WP8}, 
we have 
\begin{align}
\begin{split}
& \big\| \P_{+, \hi}\big((e^{iF_1} - e^{iF_2}) \Qdl \dx  v_2\big) \big\|_{L^{1+\kk}_TH^s_x}\\
& \quad 
\les \Big( \|\wt v \|_{L^\infty_T L^2_x}
+ \| e^{iF_1} - e^{iF_2}\|_{L^\infty_{T, x}}\big(1+ \|v_1\|_{L^\infty_T L^2_x}\big)
\Big)\| \Qdl  v_2\|_{L^\infty_TH^{s+1}_x}\\
& \quad 
 \les  \dl^{-1}(1+ \dl^{-1})
\|v_2 \|_{L^\infty_TH^s_x}
\|\wt v \|_{L^\infty_T L^2_x}.
\end{split}
\label{ILWx7}
\end{align}

Therefore, 
from the Duhamel formulation of 
\eqref{ILWx1}, 
Lemma \ref{LEM:lin}, \eqref{ILWx2}, \eqref{ILWx3}, 
\eqref{ILWx4}, 
\eqref{ILWx4a}, 
\eqref{ILWx6}, 
and 
\eqref{ILWx7}
together with \cite[the math display after (4-4)]{MP}, 
\eqref{ILWx5}, and  \eqref{WP8}, we obtain
\begin{align}
\begin{split}
\| \wt w\|_{Y^{s, \frac 12}(T)}
& \le C \| \wt w (0)\|_{H^s}\\
& \quad +
C \big(\eps + \|w_1\|_{X^{s, \frac 12}(T)}\big)\Big(\| \wt v \|_{L^\infty_T L^2_x} + \|  \wt v \|_{L^4_{T, x}} 
+ \| \wt v \|_{X^{-1, 1}(T)}\Big)
\\
&
\quad + 
C \eps  \| w_2 \|_{Y^{s, \frac 12}(T)}
\|\wt v \|_{L^\infty_TL^2_x}\\
& \quad + C \dl^{-1} (1+ \dl^{-1})
\Big( 
\eps \|\wt v \|_{L^{\infty}_TH^s_x}
+   \| v_2 \|_{L^{\infty}_TH^s_x}
\|\wt v \|_{L^{\infty}_TL^2_x}\Big)\\
& \quad 
+  2 c_0 
\Big(\|\wt v \|_{L^{\infty}_TH^s_x}
+  \| v_1 \|_{L^{\infty}_TH^s_x}
\|\wt v \|_{L^{\infty}_TL^2_x}\Big)
\end{split}
\label{ZP2}
\end{align}

\noi
for any $0 < T\le T_*$, 
where $c_0$ is as in \eqref{ILWx3}.

From \eqref{ILW3}, we have 
\begin{align}
\dt \wt v - \H \dx^2 \wt v = \Qdl\dx \wt v +  \dx \big((v_1 + v_2) \wt v\big).
\label{ZP3}
\end{align}

\noi
Then, by a slight modification of the proof of \eqref{MP1a} in 
Lemma \ref{LEM:MP1} with \eqref{WP8}, 
we have 
\begin{align}
\| \wt v \|_{X^{-1, 1}(T)}
 \les T^\ta (1 + \dl^{-1}) \|\wt v\|_{L^\infty_T L^2_x}
+  \eps \|\wt v\|_{L^4_{T, x}}
\label{ZP4}
\end{align}

\noi
for any $0 < T\le T_*$.

In the following, we estimate
the $L^p_T W^{s, q}_x$-norm of 
$\wt v = \P_+ \wt v+ \cj {\P_+ \wt v}$.
From \eqref{ILW8}
with $e^{iF} v = -i \dx e^{iF}$, 
we have 
\begin{align}
\begin{split}
\P_+ \wt v & = 
\P_{+, \LO} \wt v 
+  \P_{+, \HI} \wt v \\
& = 
\P_{+, \LO} \wt v 
- i \P_{+, \HI}(e^{-i F_1}  \wt w  )
- i \P_{+, \HI}\big((e^{-i F_1} - e^{-i F_2})  w_2  \big)\\
& \quad 
-i  \P_{+, \HI}\big(\P_{+, \hi} e^{-i F_1} \cdot \dx  \P_\lo (e^{i F_1} - e^{i F_2}) \big)\\
& \quad 
 -i  \P_{+, \HI}\big(\P_{+, \hi} (e^{-i F_1} - e^{-i F_2}) \cdot \dx \P_\lo e^{iF_2} \big)\\
& \quad 
-i  \P_{+, \HI}\big( \P_{+, \HI}e^{-i F_1} \cdot \dx \P_{-, \hi} (e^{i F_1} - e^{i F_2})\big)\\
& \quad 
 -i  \P_{+, \HI}\big( \P_{+, \HI}(e^{-i F_1} - e^{-i F_2}) \cdot \dx \P_{-, \hi} e^{iF_2}\big)\\
 & =: E_1 + \cdots +E_7.
\end{split}
\label{ZP5}
\end{align}

\noi
As pointed out in the proof of Lemma \ref{LEM:MP1}, 
\eqref{ZP5} agrees with the corresponding expression
in  \cite[p.385]{MP} (modulo constants).
In particular, 
 when (i)~$(p, q) = (\infty, 2)$ and $0 \le s \le \frac 12$
 or (ii)~$(p, q) = (4, 4)$ and $0 \le s \le \frac 14$, 
the $L^p_T W^{s, q}_x$-norm of 
$\wt v = \P_+ \wt v+ \cj {\P_+ \wt v}$
is already estimated in 
 \cite[(4-8)]{MP}, 
 except for the additional term 
 in the 
low-frequency contribution 
which is estimated as in \eqref{MP1dd}:
\begin{align}
\bigg\| J^s \int_0^t S(t-t') \P_{+, \LO} \Qdl \dx \wt v (t') dt' \bigg\|_{L^p_T L^q_x}  
\les T  \dl^{-1} \|\wt v \|_{L^\infty_T L^2_x}.
\label{ZP5a}
\end{align}

\noi
Hence, we only need to consider the case
$(p, q) = (4, 4)$ and $\frac 14 <  s \le \frac 12$.

In the following, we work with $T \in (0, T_*]$
such that  \eqref{WP8} holds.
By applying $\P_{+, \LO}$ to the Duhamel formulation of \eqref{ZP3}
and 
proceeding as in \cite[(3-17) and (3-18)]{MP} with \eqref{WP8}
and~\eqref{ZP5a}, 
given an absolute  small constant $c_1 > 0$ (to be chosen later), 
we have 
\begin{align}
\begin{split}
\|  J^s E_1\|_{L^4_{T, x}}
& \le \| \wt v(0)\|_{L^2}
+ C\big(\eps + T (1+ \dl^{-1}) \big) \|\wt v\|_{L^\infty_T L^2_x}\\
& \le \| \wt v(0)\|_{L^2}
+ (C\eps + c_1 ) \|\wt v\|_{L^\infty_T L^2_x}, 
\end{split}
\label{ZP6}
\end{align}

\noi
where the second step holds 
by possibly making $T_* = T_*(\min(\dl, 1)) > 0$ smaller.
Proceeding as in \eqref{MP1e} with~\eqref{WP8}, 
we have 
\begin{align}
\begin{split}
\|  J^s E_2\|_{L^4_{T, x}}
& \les
\big( 1+ \|v_1\|_{L^\infty_T L^2_x} \big)
 \|  \wt w  \|_{Y^{s, \frac 12}(T)}
  \les
 \|  \wt w  \|_{Y^{s, \frac 12}(T)}. 
 \end{split}
 \label{ZP6a}
\end{align}

\noi
By Lemma \ref{LEM:prod2}, \eqref{ILWx5},
Lemma \ref{LEM:L4} with \eqref{WP8}, we have 
\begin{align}
\begin{split}
\|J^s E_3\|_{L^4_{T, x}}
& \les 
 \Big(\|\wt v \|_{L^\infty_T L^2_x}
+ \| e^{-iF_1} - e^{-iF_2}\|_{L^\infty_{T, x}}
\big(1 + \|v_1\|_{L^\infty_T L^2_x}\big)\Big)
\| J^s w_2\|_{L^4_{T, x}}\\
& \les \|w_2\|_{Y^{s, \frac 12}(T)} \|\wt v \|_{L^\infty_T L^2_x}.
\end{split}
\label{ZP6b}
\end{align}

\noi
Proceeding as in  \cite[(3-14)]{MP} with \eqref{ILWx5} and \eqref{WP8}, we have
\begin{align}
\begin{split}
\|  J^s E_4\|_{L^4_{T, x}}
& \les \| v_1\|_{L^\infty_T L^2_x}
\| e^{iF_1} v_1 - e^{iF_2}v_2\|_{L^\infty_T L^2_x}\\
& \les \| v_1\|_{L^\infty_T L^2_x}
\Big( \| \wt v\|_{L^\infty_T L^2_x}
+ \| e^{iF_1} - e^{iF_2}\|_{L^\infty_{T, x}} \| v_2\|_{L^\infty_T L^2_x}\Big)\\
& \les \eps \| \wt v\|_{L^\infty_T L^2_x}.
\end{split}
\label{ZP7}
\end{align}

\noi
Similarly, 
proceeding as in  \cite[(3-14)]{MP}
with  \eqref{gauge1},  \eqref{ILWx5}, and \eqref{WP8}, we have 
\begin{align}
\begin{split}
\|  J^s E_5\|_{L^4_{T, x}}
& \les 
\|\dx \P_{+, \hi} (e^{-i F_1} - e^{-i F_2})\|_{L^\infty_T L^2_x}\| v_2\|_{L^\infty_T L^2_x}\\
& \les  \Big(\|\wt v \|_{L^\infty_T L^2_x}
+ \| e^{-iF_1} - e^{-iF_2}\|_{L^\infty_{T, x}}\|v_2\|_{L^\infty_T L^2_x}\Big)\| v_2\|_{L^\infty_T L^2_x}\\\
& \les \eps  \|\wt v \|_{L^\infty_T L^2_x}.
\end{split}
\label{ZP8}
\end{align}

\noi
Next, we treat $E_6$ and $E_7$.
Proceeding as in \eqref{ZP8}, we have 
\begin{align}
\begin{split}
\|\dx \P_{-, \hi} (e^{i F_1} - e^{i F_2})\|_{L^\infty_T L^2_x}
& \les  \|\wt v \|_{L^\infty_T L^2_x}.
\end{split}
\label{ZP10a}
\end{align}

\noi
By 
Lemma \ref{LEM:prod2} (recall that $\frac14 < s \le \frac 12$)
with \eqref{ILWx5} and \eqref{WP8}, we have 
\begin{align}
\begin{split}
 \| & e^{-i F_1}  v_1 - e^{- i F_2} v_2\|_{L^\infty_T H^{s-\frac 14}_x}
 \le \| e^{-i F_1} \wt v\|_{L^\infty_T H^{s-\frac 14}_x}
+ 
\| (e^{-i F_1}  - e^{- i F_2}) v_2\|_{L^\infty_T H^{s-\frac 14}_x}
\\
& \les  
\big(1 +  \|v_1\|_{L^\infty_T L^2_x}\big)
  \|\wt v \|_{L^\infty_T H^{s- \frac 14}_x}\\
& \quad 
+ \Big(  \|\wt v \|_{L^\infty_T L^2_x}
+ \| e^{-iF_1} - e^{-iF_2}\|_{L^\infty_{T, x}}
\big( 1+ \|v_1\|_{L^\infty_T L^2_x}\big) \Big) \| v_2 \|_{L^\infty_T H^{s- \frac 14}_x}\\
& \les 
  \|\wt v \|_{L^\infty_T H^{s}_x}
  +  \| v_2 \|_{L^\infty_T H^{s}_x}
   \| \wt v \|_{L^\infty_T L^2_x}.
\end{split}
\label{ZP10b}
\end{align}

\noi
Thus, proceeding as in 
\eqref{MP4b}-\eqref{MP4e} with \eqref{ZP10a}, \eqref{ZP10b},
and \eqref{WP8},  we
obtain
\begin{align}
\begin{split}
 & \|  J^s E_6\|_{L^4_{T, x}}
+ \| J^s E_7\|_{L^4_{T, x}}\\
& \quad \les  \big(1 + 
 \|v_1\|_{L^\infty_T L^2_x}\big)
  \|v_1 \|_{L^\infty_T H^s_x}
 \|\wt v \|_{L^\infty_T L^2_x}\\
& \quad \quad +    
\Big(  \|\wt v \|_{L^\infty_T H^{s}_x}
  +  \| v_2 \|_{L^\infty_T H^{s}_x}
   \| \wt v \|_{L^\infty_T L^2_x}\Big)
 \|v_2 \|_{L^\infty_T L^2_x}\\
& \quad \les 
  \|v_1 \|_{L^\infty_T H^s_x}
 \|\wt v \|_{L^\infty_T L^2_x}
+
\eps  \|\wt v \|_{L^\infty_T H^s_x}
+   \eps  \| v_2 \|_{L^\infty_T H^{s}_x}
   \| \wt v \|_{L^\infty_T L^2_x}.
\end{split}
\label{ZP11}
\end{align}

\noi
Therefore, putting \eqref{ZP5}, \eqref{ZP6}, 
\eqref{ZP6a}, \eqref{ZP6b}, 
\eqref{ZP7}, \eqref{ZP8},
and \eqref{ZP11}
together with   \cite[(4-8)]{MP} and~\eqref{ILWx5}, we obtain
\begin{align}
\begin{split}
\|  J^s \wt v \|_{L^p_T L^q_x}
& \les \| \wt v(0)\|_{L^2}
+ \Big(  \|w_2\|_{Y^{s, \frac 12}(T)} +   \|v_1 \|_{L^\infty_T H^s_x}
+ \eps  \| v_2 \|_{L^\infty_T H^{s}_x}
+ c_1 \Big)\|\wt v \|_{L^\infty_T L^2_x}\\
& \quad + \eps \|\wt v\|_{L^\infty_T H^s_x}
+ \|  \wt w  \|_{Y^{s, \frac 12}(T)}
\end{split}
\label{ZP12}
\end{align}

\noi
for  $0 \le s \le \frac 12$, 
$(p, q) = (\infty, 2)$ or $(4, 4)$, 
and 
 any $0 < T\le T_*$, 
 where $c_1$ is as in \eqref{ZP6}.

By collecting 
\eqref{ZP1}, 
\eqref{ZP2}, 
\eqref{ZP4}, 
and
\eqref{ZP12} for $s = 0$ with \eqref{WP8},  
we first obtain  
\begin{align}
\begin{split}
\| \wt v \|_{L^\infty_T L^2_x}
+ \|  \wt v \|_{L^4_{T,x}}
+ \| \wt w\|_{Y^{0, \frac 12}(T)}
+ \| \wt v \|_{X^{-1, 1}(T)}
& \les \| \wt v(0)\|_{L^2}
\end{split}
\label{ZP12a}
\end{align}

\noi
for any $0 < T \le T_*$, 
provided that 
we choose 
(i) $\eps_0 = \eps_0(\min(\dl, 1))> 0$ in \eqref{WP8}
sufficiently small such that 
\begin{align}
\eps_0 \dl^{-1} (1+ \dl^{-1}) \ll 1
\label{ZP12b}
\end{align}

\noi
and 
(ii)~$c_0 > 0$ in \eqref{ILWx3} and \eqref{ZP2}
and $c_1> 0$  in \eqref{ZP6}
sufficiently small, 
which can be guaranteed by 
choosing
 $T_* = T_*(\min(\dl, 1)) > 0$  sufficiently small.
Finally, from 
\eqref{ZP1}, 
\eqref{ZP2}, 
\eqref{ZP4}, 
and
\eqref{ZP12} with \eqref{ZP12a}, 
we conclude that, with $M$ as in \eqref{WP8}, 
\begin{align}
\begin{split}
\|  J^s \wt v \|_{L^\infty_T L^2_x}
+ \|  J^s \wt v \|_{L^4_{T,x}}
+ \| \wt w\|_{Y^{s, \frac 12}(T)}
+ \| \wt v \|_{X^{-1, 1}(T)}
& \le C(M) \| \wt v(0)\|_{H^s}
\end{split}
\label{ZP13}
\end{align}

\noi
for any $0 < T \le T_*$, 
provided that  $\eps_0, c_0 > 0$
are sufficiently small as above.

\subsection{Well-posedness}
\label{SUBSEC:LWP3}

Once we have \eqref{ZP13}, 
local well-posedness of 
the renormalized ILW \eqref{ILW2}
(and hence of the original ILW equation \eqref{ILW1})
follows
from a standard argument.
Given $0 \le s \le \frac 12$, 
let $u_0 \in H^s_0(\T)$
with $2 \| u_0\|_{L^2} \le \eps_0 = \eps_0(\min(\dl, 1))$
and consider the sequence
$\{u_{0, k}\}_{k \in \N}$ of smooth mean-zero initial data
$u_{0, k} = \Pi_{\le k} u_0$, 
where $\Pi_{\le k}$
is the frequency projector onto the (spatial) frequencies $\{n \in \Z: |n| \le k\}$.
Let $v_k$ be the smooth global solution 
to \eqref{ILW2}
with $v_k|_{t = 0} = u_{0, k}$.
Then, it follows from \eqref{ZP13}
that $v_k$ converges to a limit $v$
in $C([0,T_*]; H^s(\T)) \cap 
L^4([0,T_*]; W^{s, 4}(\T))$
and that the associated gauged function $w_k = w_k(v_k)$
(defined as in \eqref{gauge6})
converges to a limit $w$ in $Y^{s, \frac 12}(T_*)$, 
where 
 $T_* = T_*(\min(\dl, 1)) > 0$  
 is as in the previous subsections.
It is easy to verify that the limit $v$
satisfies the equation~\eqref{ILW2}
with $v|_{t = 0} = u_0$
and that the limits $v$ and $w$ satisfy~\eqref{gauge6}.
In the current periodic setting, 
uniqueness of solutions and local Lipschitz continuity
of the solution map for the renormalized ILW \eqref{ILW2} follow from~\eqref{ZP13}.
Furthermore, 
by noting that the local existence time 
$T_* = T_*(\min(\dl, 1)) > 0$
depends only on $\dl > 0$
(as long as we choose $\eps_0$ (= the upper bound on the $L^2$-norm 
of initial data) sufficiently small such that 
\eqref{ZP12b} is satisfied), 
we conclude global well-posedness
of \eqref{ILW2} from 
the conservation of the $L^2$-norm.
Finally, 
by undoing the Galilean transform \eqref{gauge0}, 
we obtain global well-posedness of the original ILW equation \eqref{ILW1}.
This concludes the proof of Theorem \ref{THM:1}.

\begin{remark}
\label{REM:scaling} \rm

In this section, we proved local and global well-posedness
 of the renormalized ILW \eqref{ILW2}
 (and then of ILW \eqref{ILW1})
with the small $L^2$-norm assumption.
In the general case, 
we apply a scaling argument and reduce the situation 
to the small $L^2$-norm case as in \cite{CKSTT, Moli1, Moli2, MP}.
Given $\ld \ge 1$, let $\S_\ld$ be
the  $\dot H^{-\frac 12}$-invariant scaling operator defined by 
\begin{align}
\S_\ld(u) (t, x) = \ld^{-1} u (\ld^{-2}t, \ld^{-1} x).
\label{S1}
\end{align}

\noi
Recall 
that the BO equation \eqref{BO1} on $\R$
in invariant under the action of $\S_\ld$.
While the renormalized ILW \eqref{ILW2} on $\R$ does not 
enjoy any  scaling symmetry, 
a direct computation shows that if $v$ is a solution 
to \eqref{ILW2} on $\R$ with
 the depth parameter $\dl$, 
then $v_\ld = \S_\ld(v)$
is a solution to 
\eqref{ILW2}  on $\R$ with the depth parameter $\ld \dl$:
\begin{align}
\dt v_\ld -
 \mathcal{T}_{\ld \dl} \dx^2 v_\ld =  \dx (v_\ld^2) .
\label{ILW2x}
\end{align}

\noi
Namely, {\it the family of the renormalized ILW equations
with depth parameters $0 < \dl < \infty$ remains invariant
under the action of $\S_\ld$}.
In the periodic setting, 
if
$v$ satisfies  \eqref{ILW2} on $\T$
then $v_\ld = \S_\ld(v)$
satisfies \eqref{ILW2x}
on the dilated torus 
$\T_\ld = \R/(2\pi \ld \Z)$.
With $u_0 = v(0)$, 
the scaled initial data for $v_\ld$ satisfying \eqref{ILW2x} is given by 
$u_{0, \ld}(x) = \ld^{-1} u_0 ( \ld^{-1} x)$.
Noting that 
\begin{align*}
\| u_{0, \ld}\|_{L^2(\T_\ld)} = \ld^{-\frac 12} \|u_0\|_{L^2(\T)}, 
\end{align*}

\noi
we can choose $\ld \gg 1$ such that 
$\| u_{0, \ld}\|_{L^2(\T_\ld)} \le \eps_0 = \eps_0(\min(\dl, 1))$.
We point out that 
we can choose 
$\eps_0(\min(\dl, 1))$
independently of  $\ld \ge1$
(since $\ld \dl \ge \dl$).
As pointed out in \cite[Section 7]{MP},
all the estimates hold true 
on $\T_\ld$ with constants independent of $\ld \ge 1$.
Here, the most important (and well-known) point
is that the periodic $L^4$-Strichartz
estimate (Lemma~\ref{LEM:L4}) holds
true
on $\T_\ld$ with a constant independent of $\ld \ge 1$.
With this observation, 
we can repeat the argument presented in this section
and prove global well-posedness of \eqref{ILW2x}
in $H^s(\T_\ld)$ with the small  $L^2$-norm assumption.
Then, by undoing the scaling, we conclude
global well-posedness of \eqref{ILW2} (and hence of \eqref{ILW1})
in $H^s(\T)$ without the small  $L^2$-norm assumption.

\end{remark}

\begin{remark}\rm
The bound \eqref{ILWx5}, coming from  \eqref{diff1} and \eqref{diff2}, 
played a crucial role in the argument on $\T$ presented above, 
yielding local Lipschitz continuity of the solution map in the periodic case.
As pointed out in Remark \ref{REM:diff},  however, 
  \eqref{diff1} and \eqref{diff2} (and hence \eqref{ILWx5})  do not hold 
in the real line case.
This is the primary reason for the failure of local uniform continuity
of the solution map for ILW on $\R$ constructed in Theorem \ref{THM:1}.
Nonetheless, under 
the extra assumption~\eqref{diff3}, 
one can still prove~\eqref{diff1} and \eqref{diff2} on~$\R$
(but  the argument uses the equation; see \cite[Lemma 4.1]{MP}
in the case of the BO equation), 
giving local Lipschitz continuity
of the solution map restricted to the class 
\eqref{diff3} (say, for fixed $v_1(0)$)
consisting of elements with the identical low-frequency parts.
See Section 4 in \cite{MP}
for a further discussion.

\end{remark}

\section{Deep-water limit of  ILW in $L^2$}
\label{SEC:conv}

In this section, 
by slightly modifying the argument in Section \ref{SEC:LWP}, 
we present the proof of Theorem \ref{THM:2}.
As in the previous sections, we only consider the periodic case.
Moreover, we only consider the case with sufficiently small $L^2$-norms.

\medskip

\noi
$\bullet$ {\bf Part  1:}
In this part, we consider the mean-zero case.
Namely, 
given $0 \le s \le \frac 12$, fix a mean-zero function  $u_0 \in H^s_0(\T)$
and 
let  $\{u_{0, \dl}\}_{1 \le \dl <  \infty}$
be a net in $H^s_0(\T)$
(of mean-zero functions)
such that 
$u_{0, \dl}$ converges to $u_0$ in $H^s(\T)$ as $\dl \to \infty$.

We first establish convergence of solutions to the renormalized ILW \eqref{ILW2}
to the associated solution to the BO equation \eqref{BO1}.
Let $u$ be the global solution to the BO equation \eqref{BO1}
with $u|_{t = 0} = u_0$
constructed in \cite{MP}, 
and 
 let $v_\dl$,  $1 \le \dl < \infty$, be the global solution to \eqref{ILW2} with 
$v_\dl|_{t = 0} = u_{0, \dl}$
constructed in Theorem \ref{THM:1}.
Recalling that $\mathcal{Q}_\infty = 0$, 
the renormalized ILW~\eqref{ILW2} formally reduces to the BO equation \eqref{BO1}
when $\dl = \infty$.
In view of this, we set $v_\infty = u$ and $u_{0, \infty} = u_0$.
For each $1 \le \dl \le \infty$, 
let  $F_\dl$ be the mean-zero primitive of $v_\dl$
as defined in \eqref{gauge1}, 
and 
let 
$W_\dl$ and $w_\dl$ be the associated gauged functions 
as in \eqref{gauge4} and \eqref{gauge6}.

Let $N^s_T$ be as in \eqref{WP1}.
Then, the a priori bound 
\eqref{WP5}
holds uniformly in  $1 \le \dl \le \infty$ (including $\dl = \infty$): 
\begin{align}
N^s_T(v_\dl) 
& \les \big( 1 + \|u_{0, \dl}\|_{L^2}\big)\|u_{0, \dl}\|_{H^s}
+ 
N^0_T(v_\dl)\big(1+ N^0_T(v_\dl)\big)
N^s_T(v_\dl)
\label{XP1}
\end{align}

\noi
for any $0 < T \le T_*$, 
where $T_* > 0$ is now independent of  $1 \le \dl \le \infty$.
By assuming that 
$\|u_{0, \infty}\|_{L^2} = \|u_0\|_{L^2} \le \eps$
for sufficiently small $\eps > 0$, 
it follows from 
the convergence of $u_{0, \dl}$ to $u_{0, \infty}$ in $H^s(\T)$
that there exists
$\dl_0 = \dl_0(\eps) \ge 1$ such that 
$\|u_{0, \dl}\|_{L^2} \les \eps$
for any $\dl \ge \dl_0$.
Then, 
from \eqref{XP1}, we have 
\begin{align}
N^0_{T_*}(v_\dl) 
 \les  \eps
\qquad \text{and}\qquad
N^s_{T_*}(v_\dl) 
& \les \|u_{0, \dl}\|_{H^s}\le M 
\label{XP3}
\end{align}

\noi
for some $M > 0$, 
uniformly in $\dl_0 \le \dl \le \infty$.

Let $\wt v_\dl = v_\dl - v_\infty$, 
$\wt W_\dl = W_\dl - W_\infty$, and $\wt w_\dl = w_\dl - w_\infty$.
The difference 
$\wt w_\dl$ satisfies 
\begin{align}
\begin{split}
\dt &  \wt w_\dl - \H \dx^2 \wt w_\dl \\
& = - 2 \dx \P_{+, \hi}(W_\dl \P_- \dx \wt v_\dl)
- 2 \dx \P_{+, \hi}(\wt W_\dl \P_- \dx v_\infty)\\
& \quad 
- 2 \dx \P_{+, \hi}(\P_\lo e^{iF_\dl} \P_- \dx \wt v_\dl) 
- 2 \dx \P_{+, \hi}\big(\P_\lo (e^{iF_\dl} - e^{iF_\infty} )\P_- \dx v_\infty\big) \\
& \quad 
+ i \dx\P_{+, \hi}(e^{iF_\dl}  \Qdl v_\dl) 
+ 0 \\
& \quad 
- i  \P_0( v_\dl^2) \wt w_\dl
- i  \P_0( v_\dl^2 - v_\infty^2 ) w_\infty.
\end{split}
\label{XP4}
\end{align}

\noi
Compare this with \eqref{ILWx1}, 
where the difference appears only 
in the fifth term and the sixth term (which is $0$ in \eqref{XP4}).
In estimating  the fifth term
on the right-hand side of  \eqref{XP4}, we instead use \eqref{MP2f}.
Hence, from the discussion in Subsection \ref{SUBSEC:LWP2}
leading to \eqref{ZP2}
with \eqref{XP3}, we have
\begin{align}
\begin{split}
\| \wt w_\dl\|_{Y^{s, \frac 12}(T_*)}
& \le C \| \wt w_\dl (0)\|_{H^s}\\
& \quad + C\big(\eps + \|w_\dl\|_{X^{s, \frac 12}(T_*)}\big)
\Big(\| \wt v_\dl \|_{L^\infty_{T_*} L^2_x} + \|  \wt v_\dl \|_{L^4_{T_*, x}} 
+ \| \wt v_\dl \|_{X^{-1, 1}(T_*)}\Big)
\\
&
\quad + 
C \eps  \| w_\infty \|_{Y^{s, \frac 12}(T_*)}
\|\wt v_\dl \|_{L^\infty_{T_*}L^2_x}
+ 
T_*^\ta M \dl^{-1} ,  
\end{split}
\label{XP5}
\end{align}

\noi
uniformly in $\dl_0 \le \dl \le \infty$.
From \eqref{BO1} and \eqref{ILW3}, we have 
\begin{align}
\dt \wt v_\dl - \H \dx^2 \wt v_\dl = \Qdl\dx  v_\dl +  \dx \big((v_\dl + v_\infty) \wt v_\dl\big).
\label{XP6}
\end{align}

\noi
Then, by a slight modification of the proof of \eqref{MP1a} in 
Lemma \ref{LEM:MP1} with \eqref{XP3}, 
we have 
\begin{align}
\begin{split}
\| \wt v_\dl \|_{X^{-1, 1}(T_*)}
&  \les  \|\wt v_\dl\|_{L^\infty_{T_*} L^2_x}
+  \eps \|\wt v_\dl\|_{L^4_{T_*, x}}
+ \dl^{-1} \| v_\dl\|_{L^\infty_{T_*} L^2_x}\\
&  \les  \|\wt v_\dl\|_{L^\infty_{T_*} L^2_x}
+  \eps \|\wt v_\dl\|_{L^4_{T_*, x}}
+ \eps  \dl^{-1} , 
\end{split}
\label{XP7}
\end{align}

\noi
uniformly in $\dl_0 \le \dl \le \infty$.

Lastly,
note that 
 \eqref{XP6} is different from 
\eqref{ZP3} only on the first term on the right-hand side of \eqref{XP6}
(under the identification of $v_1$ and $v_2$
with $v_\dl$ and $v_\infty$, respectively).
Namely, 
 the bound  \eqref{ZP12}
holds except for 
$c_1 \|\wt v \|_{L^\infty_T L^2_x}$
in  \eqref{ZP12}, coming from \eqref{ZP5a}
and \eqref{ZP6}.
In this case, we instead apply
\eqref{MP1dd} with \eqref{XP3}
to estimate the contribution from the first term on the right-hand side
of~\eqref{XP6}.
This yields
\begin{align}
\begin{split}
\|  J^s \wt v_\dl \|_{L^p_{T_*} L^q_x}
& \les \| \wt v_\dl(0)\|_{L^2}
+ \Big( \|w_\infty\|_{Y^{s, \frac 12}(T_*)} +   \|v_\dl \|_{L^\infty_{T_*} H^s_x}
+ \eps  \| v_\infty \|_{L^\infty_{T_*} H^{s}_x}\Big)\|\wt v_\dl \|_{L^\infty_{T_*} L^2_x}\\
& \quad + \eps \|\wt v_\dl\|_{L^\infty_{T_*} H^s_x}
+ \|  \wt w_\dl  \|_{Y^{s, \frac 12}(T_*)}
+ \eps \dl^{-1}
\end{split}
\label{XP8}
\end{align}

\noi
for  $0 \le s \le \frac 12$
and 
$(p, q) = (\infty, 2)$ or $(4, 4)$, 
uniformly in $\dl_0 \le \dl \le \infty$, 
where the last term represents
the contribution from the first term on the right-hand side
of \eqref{XP6}.

Therefore, putting \eqref{XP5}, \eqref{XP7}, and \eqref{XP8}
together
with \eqref{ZP1} (for $\wt w_\dl(0)$ and $\wt v_\dl(0)$)
and proceeding as in Subsection \ref{SUBSEC:LWP2}
(first for $s= 0$ as in \eqref{ZP12a} and then for $s$ as in \eqref{ZP13}),
we obtain
\begin{align}
\begin{split}
& \|  J^s \wt v_\dl \|_{L^\infty_{T_*} L^2_x}
+ \|  J^s \wt v \|_{L^4_{T_*,x}}
+ \| \wt w_\dl\|_{Y^{s, \frac 12}(T_*)}
+ \| \wt v_\dl \|_{X^{-1, 1}(T_*)}\\
& \quad \le C(M) \| \wt v_\dl(0)\|_{H^s} +  C(M) \dl^{-1}, 
\end{split}
\label{XP9}
\end{align}

\noi
uniformly in $\dl_0 \le \dl \le \infty$, 
provided that  $T_*  > 0$  sufficiently small.
(Recall that 
$T_* > 0$ is independent of  $1 \le \dl \le \infty$.)
Here, the last term on the right-hand side of \eqref{XP9}
comes from the last term in~\eqref{XP7}.
Since the right-hand side of \eqref{XP9} tends to~0 as $\dl \to \infty$, 
we conclude that, as $\dl \to \infty$,  
the solution $v_\dl$ to \eqref{ILW2}
converges to the solution $u = v_\infty$
to \eqref{BO1}
in $C([0, T_*]; H^s(\T))\cap L^4([0, T_*]; W^{s, 4}(\T))$.
Given any $T_0 \gg 1$, 
we can  iterate this local-in-time convergence argument
and  conclude  convergence on the entire interval $[0, T_0]$.

Next, we transfer this convergence back to the original ILW 
equation \eqref{ILW1}.
Let $u_\dl$ be the global solution to  \eqref{ILW1}
with $u_\dl|_{t = 0} = u_{0, \dl}$.
Then, from \eqref{gauge0}, we have
$u_\dl = \tau_{-\dl^{-1}} v_\dl
= \tau(-\dl^{-1},  v_\dl).$ 
Letting $u$ be the global solution to the BO equation \eqref{BO1}
with $u|_{t = 0} = u_0$ as above, 
it follows from  the triangle inequality that
\begin{align}
\begin{split}
\| u - u_\dl\|_{C_{T_0} H^s_x}
& \le \| u - \tau(-\dl^{-1}, u)\|_{C_{T_0} H^s_x}
+ 
\| \tau(-\dl^{-1}, u) - u_\dl\|_{C_{T_0} H^s_x}\\
& \le \| u - \tau(-\dl^{-1}, u)\|_{C_{T_0} H^s_x}
+ 
\| u - v_\dl\|_{C_{T_0} H^s_x}.
\end{split}
\label{XP11}
\end{align}

\noi
By noting that $u = \tau(0, u)$, 
it follows from Lemma \ref{LEM:trans}
that the first term on the right-hand side of \eqref{XP11}
tends to $0$ as $\dl \to \infty$, 
while the second term tends to $0$ as $\dl \to \infty$
in view of the discussion above.
This proves
convergence of $u_\dl$
to $u$ in $C([0, T_0]; H^s(\T))$
for any $T_0 > 0$.

\medskip

\noi
$\bullet$ {\bf Part  2:}
In this step, we consider the general case without the mean-zero assumption.
Given    $u_0 \in H^s(\T)$, 
let  $\{u_{0, \dl}\}_{1 \le \dl <  \infty}$
be a net in $H^s(\T)$
such that 
$u_{0, \dl}$ converges to $u_0$ in $H^s(\T)$ as $\dl \to \infty$.
Note that  the spatial mean $\mu(u_{0, \dl})$ of $u_{0, \dl}$
converges to the spatial mean $\mu(u_0)$ of $u_0$ as $\dl \to \infty$.

Let $u$ be the global solution to the BO equation \eqref{BO1}
with $u|_{t = 0} = u_0$
constructed in~\cite{MP}, 
and 
let $u_\dl$,  $1 \le \dl < \infty$, be the global solution to \eqref{ILW1} with 
$u_\dl|_{t = 0} = u_{0, \dl}$
constructed in Theorem \ref{THM:1}.
Then, by setting
$\wt u = \G_{u_0}(u)$ and $\wt u_\dl = \G_{u_{0, \dl}}(u_\dl)$, 
where $\G$ is the  Galilean transform defined in \eqref{gaugex}, 
we see that 
 $\wt u$ (and $\wt u_\dl$, respectively) is  the global solution to  \eqref{BO1}
with the mean-zero initial data $\wt u |_{t = 0} = \wt u_0 = u_0 - \mu(u_0)$
(and  to \eqref{ILW1}
with the mean-zero initial data $\wt u_\dl |_{t = 0} = \wt u_{0, \dl} = u_{0, \dl} - \mu(u_{0, \dl})$, 
respectively).
Hence, from Part 1, 
we see that, given any $T_0 > 0$,  $\wt u_\dl$ converges to $\wt u$
in $C([0, T_0]; H^s(\T))$
as $\dl \to \infty$.

With $\tau_h$ as in 
\eqref{trans1}, we have
\begin{align*}
u = \tau (2\mu(u_0), \wt u) + \mu(u_0)\qquad \text{and}
\qquad u_\dl = \tau (2\mu(u_{0, \dl}), \wt u_\dl) + \mu(u_{0, \dl}).
\end{align*}

\noi
Therefore, in view of the convergence of 
$\mu(u_{0, \dl})$ and $\wt u_\dl$  to $\mu(u_0)$ and $\wt u$, respectively, 
we conclude from the joint continuity of $\tau$
(Lemma \ref{LEM:trans}) that 
 $ u_\dl$ converges to $ u$
in $C([0, T_0]; H^s(\T))$
as $\dl \to \infty$.
This concludes the proof of Theorem \ref{THM:2}.

\begin{ackno}\rm
This material is based upon work supported by the Swedish
Research Council under grant no.~2016-06596
while 
the first, third, and fourth authors were in residence at 
Institut  Mittag-Leffler in Djursholm, Sweden
during the 
program ``Order and Randomness in Partial Differential Equations''
in Fall, 2023.
A.C., G.L., and T.O.~were supported by the European Research Council
(grant no. 864138 ``SingStochDispDyn'').
G.L.~was also supported by the EPSRC New Investigator Award (grant no.~EP/S033157/1).
D.P. was supported by a Trond Mohn Foundation grant.
\end{ackno}

\end{document}